\documentclass[12pt]{amsart}

\calclayout

\usepackage{amsmath,amsthm,amssymb}
\usepackage{mathrsfs}

\usepackage{colonequals}
\usepackage{enumitem}
\usepackage{graphicx}
\usepackage[justification=centering]{caption}

\usepackage[colorlinks=true,allcolors=blue,pdftitle={Bounded vorticity}]{hyperref}

\usepackage[initials,nobysame,alphabetic,msc-links,backrefs]{amsrefs}
\usepackage{color}

\newtheorem{theorem}{Theorem}
\newtheorem{theoremm}{Theorem}[section]

\newtheorem{corollary}{Corollary}[section]
\newtheorem{proposition}{Proposition}[section]
\newtheorem{lemma}[proposition]{Lemma}

\newtheorem{remark}{Remark}[section]
\newtheorem{condition}{Condition}[section]
\newtheorem{defi}{Definition}[section]

\numberwithin{equation}{section}

\DeclareMathOperator{\diver}{div}
\DeclareMathOperator{\curl}{curl}
\DeclareMathOperator{\var}{var}
\DeclareMathOperator{\R}{R}%notation for the ratio

\newcommand{\de}{\delta}

\newcommand{\ep}{\varepsilon}
\newcommand{\ve}{\varepsilon}
\newcommand{\lep}{{|\log\ep|}}
\newcommand{\ga}{\Gamma}
\newcommand{\lga}{|\ga_0|}
\newcommand{\al}{\alpha}

\newcommand{\N}{\mathbb N}
\newcommand{\NN}{\mathcal N}
\newcommand{\RR}{\mathbb R}
\newcommand{\RRR}{\mathfrak R}
\newcommand{\C}{\mathbb C}

\newcommand{\CC}{\mathcal C}

\newcommand{\F}{\mathcal F}
\renewcommand{\H}{\mathcal H}

\renewcommand{\S}{\mathcal S}
\newcommand{\T}{\mathcal T}

\newcommand{\GG}{\mathfrak G}

\def\indic{{\mathbf 1}}

\renewcommand{\CC}{\mathscr C}
\newcommand{\E}{\mathscr E}

\renewcommand{\t}{\mathbf t}

\newcommand{\D}{\mathrm D}

\newcommand{\ex}{\mathrm{ex}}
\newcommand{\loc}{\mathrm{loc}}
\newcommand{\good}{\mathrm{good}}
\newcommand{\bad}{\mathrm{bad}}

\renewcommand{\(}{\left(}
\renewcommand{\)}{\right)}

\renewcommand{\u}{\mathbf{u}}
\newcommand{\A}{\mathbf{A}}

\def\pr#1{\left\langle #1\right\rangle}

\def\hal{{\frac12}}

\def\hal{{\frac12}}
\def\nab{{\nabla}}
\def\hci{{H_{c_1}}}
\def\tga{{\widetilde\ga}}

\definecolor{green1}{rgb}{0.0, 0.5, 0.0}

\DeclareMathOperator{\curlcurl}{curl^2}

\title[Bounded vorticity for 3D Ginzburg--Landau]{Bounded vorticity for the 3D Ginzburg--Landau model and an isoflux problem}

\author{Carlos Rom\'{a}n}
\address{Facultad de Matem\'aticas e Instituto de Ingenier\'ia Matem\'atica y Computacional, Pontificia Universidad Cat\'olica de Chile, Vicu\~na Mackenna 4860, 7820436 Macul, Santiago, Chile}
\email{carlos.roman@mat.uc.cl}

\author{Etienne Sandier}
\address{Universite Paris-Est, LAMA - CNRS UMR 8050, 61, Avenue du General de Gaulle, 94010 Creteil, France}
\email{sandier@u-pec.fr}

\author{Sylvia Serfaty}
\address{Courant Institute of Mathematical Sciences, New York University, 251 Mercer St., New York, NY 10012}
\email{serfaty@cims.nyu.edu}

\date{October 28, 2022}
\begin{document}
\begin{abstract} We consider the full three-dimensional Ginzburg--Landau model of superconductivity with applied magnetic field, in the regime where the intensity of the applied field is close to the ``first critical field" $H_{c_1}$ at which vortex filaments appear, and in the asymptotics of a small inverse Ginzburg--Landau parameter $\ep$. This onset of vorticity is directly related to an ``isoflux problem" on curves (finding a curve that maximizes the ratio of a  magnetic flux by its length), whose study was initiated in \cite{Rom2} and which we continue here. By assuming a nondegeneracy condition for this isoflux problem, which we show holds at least for instance in the case of a ball, we prove that if the intensity of the applied field remains below $\hci+ C \log \lep$, the total vorticity remains bounded independently of $\ep$,  with vortex lines concentrating near the maximizer of the isoflux problem, thus extending to the three-dimensional setting a two-dimensional result of \cite{SanSer2}.    We finish by showing an improved estimate on the value of $\hci$ in some specific simple geometries. 
\end{abstract}
\maketitle 

\section{Introduction}

\subsection{Setup and problem}
We are interested in the full three-dimensional Ginzburg--Landau model of superconductivity, which after nondimensionalization of the constants, can be written as 
\begin{equation}
\label{modgl}
GL_\ep (u, A):= \frac12 \int_\Omega |\nab_A u|^2 + \frac1{2\ep^2}(1-|u|^2)^2+ \hal \int_{\RR^3} |H-H_{\ex}|^2.\end{equation}
Here $\Omega $ represents the material sample, we assume it to be a bounded simply connected subset of $\RR^3$ with regular boundary. 
The function $u : \Omega \to \C$ is the ``order parameter", representing the local state of the material in this macroscopic theory ($|u|^2\le 1$ indicates the local density of superconducting electrons), while the vector-field $A: \RR^3 \to \RR^3$ is the gauge of the magnetic field, and the magnetic field induced inside the sample and outside is $H:= \nab \times A$. The covariant derivative $\nab_A$ means $\nab - i A$. The vector field $H_{\ex}$ here represents an applied magnetic field and we will assume that $H_{\ex}= h_{\ex} H_{0,\ex}$ where $H_{0,\ex} $ is a fixed vector field and $h_{\ex}$ is a real parameter that can be tuned. 
Finally, the parameter $\ep>0$ is the inverse of the so-called Ginzburg--Landau parameter $\kappa$, a dimensionless ratio of all material constants, and that depends only on the type of material. In our mathematical analysis of the model, we will study the asymptotics of $\ep \to 0$ (also called ``London limit" in physics) which corresponds to extreme type-II superconductors. 

The Ginzburg--Landau model is known to be a $\mathbb U(1)$-gauge theory. This means that all
the meaningful physical quantities are invariant under the gauge transformations
$u \to  u e^{i\Phi}$, $ A \to A + \nab \Phi$ where 
 $\Phi$  is any regular enough real-valued function. The Ginzburg--Landau energy and its
associated free energy
\begin{equation} \label{freee}F_\ep(u, A): = \frac12 \int_\Omega |\nab_A u|^2 + \frac1{2\ep^2}(1-|u|^2)^2+ \hal \int_{\RR^3} |H|^2
\end{equation} are gauge invariant, as well as the density of superconducting Cooper pairs $|u|^
2$, the induced magnetic field $H$, and the vorticity defined below. 

Such type-II superconductors are known to exhibit, as a function of the intensity of the applied field, phase transitions with the occurrence of {\it vortex-lines}: these are zeroes of the order parameter function $u$ around which $u$ has a nontrivial winding number, or degree (the rotation number of its phase, essentially). 

For more on this and the physics background on the model, we refer to the standard texts \cites{SSTS,DeG,Tin}. For a partial rigorous derivation of \eqref{modgl} from a microscopic quantum theory, i.e. the Bardeen--Cooper--Schrieffer theory \cite{BCS},  we refer to \cite{FHSS}.
We also note that rotating superfluids and rotating Bose--Einstein condensates can be described through a very similar Gross--Pitaevskii model, which no longer contains the gauge $A$, and where the  applied field $ H_\ex$ is replaced by a rotation vector whose intensity can be tuned. As seen in prior works, such as \cites{Ser3,aftalionjerrard,aftalionbook,BalJerOrlSon2} and references therein, in the regime of low enough rotation  these models can be treated  with the same techniques as those developed for Ginzburg--Landau.

Our goal is to contribute to the analysis and description of the vortex lines when they appear, which is for $h_\ex$ near the first critical field $\hci$, continuing the line of work of \cite{Rom2}. In particular we provide a more precise expansion of $\hci$ (as $\ep \to 0$) than previously known, and show that below $\hci + C\log \lep$, the vorticity remains bounded. The characterization of $\hci$ and the onset of vortices is also directly related to an ``isoflux problem", first introduced (as far as we know) in \cite{Rom2}, which we analyze in more detail. We believe this problem to be of independent interest. 

The two-dimensional version of the Ginzburg--Landau model, in which vortices are essentially points instead of lines, was studied in details in the mathematics literature, in particular the questions that we address here were entirely solved, with precise description of the first critical field, number and location of the first vortices  and derivation of their effective interaction energy. We refer to the book \cite{SanSerBook} and references therein. The present paper corresponds to the three-dimensional analogue of \cite{SanSer2}, see also \cite{SanSerBook}*{Chapter 9},   while our follow up paper \cite{RomSanSer} will address the derivation of an interaction energy for the vortex lines. 

In contrast with the two dimensional case, the analysis of the three-dimensional (hence more physical) full Ginzburg--Landau model is more recent, but was preceded by many works analyzing vortex lines in a simplified Ginzburg--Landau model without gauge \cites{Riv,LinRiv1,LinRiv2,BetBreOrl,San}, more precisely the three-dimensional version of the model studied in \cite{BetBreHel}. We also refer to \cites{ConJer,DavDelMedRod} for some recent developments on this model.
Most related to our results are the works on vortex lines in  the full  three-dimensional Ginzburg--Landau  of Alama-Bronsard-Montero \cite{AlaBroMon} who first derived $\hci $ and the onset of the first vortex in the case of a ball, Baldo-Jerrard-Orlandi-Soner \cite{BalJerOrlSon2} who derive a mean-field model for many vortices and the main order of the first critical field in the general case, and \cite{Rom2} who gives a more precise expansion of $\hci$ in the general case, and more significantly, proved that global minimizers have no vortices below $\hci$, while they do above this value. One may also point out \cite{JerMonSte} who construct locally minimizing solutions.

One of the main difficulties in 3D compared to 2D is that the vortices have a geometry, they are a priori nonregular curves, and they need to be described via tools from geometric measure theory, in particular currents. 
Several approaches have been put forward to analyze the dependence of the energy $GL_\ep$ on the vortex curves, in particular those of \cites{JerSon,AlbBalOrl2}. As in \cite{Rom2}, the one we will use is the approach of \cite{Rom} which has the major advantage of being $\ep$-quantitative hence more precise. The results of \cite{Rom} allow to approximate the true vortices by  lines which are Lipschitz (and piecewise straight) and around which concentrates an energy at least proportional to their degree times their length.  For a precise statement, see Theorem~\ref{theorem:epslevel} below.  Vortices are thus energetically costly on the one hand, but on the other hand energetically advantageous due to a magnetic term present in the energy and proportional to the intensity of $H_{\ex}$. This energy gain corresponds to the value of the magnetic flux through the vortex or rather the loop formed by the  vortex line on the one hand, and any curve lying on $\partial \Omega$ that allows to close it  (see Figure~\ref{figure0})
of a certain fixed magnetic field $B_0$ that is normal on the boundary. This magnetic field $B_0$ is constructed from $H_{0,\ex}$ and introduced in \cite{Rom2}, see its precise definition below.

\begin{figure}
\centering
\includegraphics[scale=0.6]{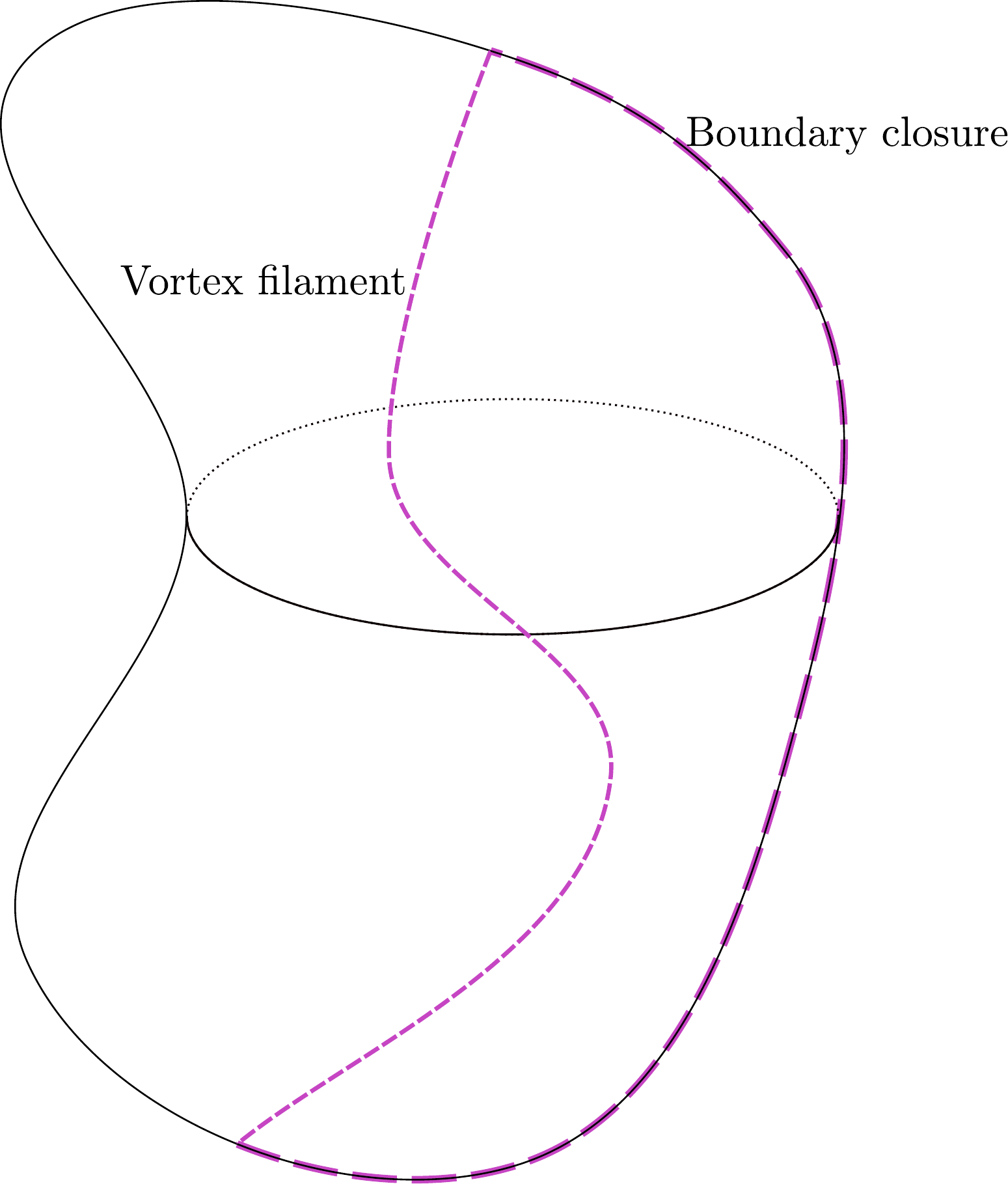}
\caption{Vortex filament and a closing curve.}
\label{figure0}
\end{figure}

 As a result of this energy competition, when $h_{\ex}$ is large enough, more precisely when it exceeds $\hci$, vortices become overall favorable. 
   This can be understood more precisely  via an energy splitting  introduced in \cite{Rom2}, see below.

The competition is thus between the flux gain, which favors vortices that enclose the largest possible domain transversal to the magnetic field  $B_0$, and the cost (proportional to length) term, which favors short vortices. Here, again, one has to understand a vortex as a line which can be completed into a loop by some arbitrary path on the boundary, this path not contributing towards the length cost. 
Such a problem is a three-dimensional (more complicated) analogue of an isoperimetric problem, with area replaced by flux with respect to a fixed vector field. By analogy with ``isoperimetric problem" we call it an ``isoflux problem". We believe it is a natural problem, of independent interest, since the same questions (existence of minimizers, regularity...) as for standard isoperimetric problems can be asked.

Let us first start by defining the isoflux problem more precisely.

\subsection{The isoflux problem}
Given a domain $\Omega \subset \RR^3$, 
 we  let $\NN$ be the space of normal $1$-currents supported in $\overline\Omega$, with boundary supported on $\partial\Omega$. We always denote 
by $|\cdot |$ the mass of a current. Recall that normal currents are currents with finite mass whose boundaries have finite mass as well.
We also let $X$ denote the class of currents in $\NN$ which are simple oriented Lipschitz curves. 
An element of $X$  must either be  a loop contained in $\overline\Omega$ or have its two endpoints on $\partial \Omega$. Given $\gamma\in(0,1]$, we let $C_T^{0,\gamma}(\Omega)$ denote the space of vector fields $B \in C^{0,\gamma}(\Omega)$ such that $B\times\vec\nu=0$ on $\partial \Omega$, where hereafter $\vec\nu$ is the outer unit normal to $\partial\Omega$.  The symbol $^*$ will denote its dual space. 

Such a $B$ may also be interpreted as 2-form, we will not distinguish the two in our notation.

For any vector field $B\in C_T^{0,1}(\Omega,\RR^3)$, and any  $\ga\in\NN$ we denote by  $\pr {B,\ga}$ the value of $\ga$ applied to $B$, which corresponds to the circulation of the vector field $B$ on $\ga$ when $\ga$ is a curve. If $\mu$ is a 2-form then $\pr {B,\mu}$ is the scalar product of the 2-form $\mu$ and $B$, seen as  a 2-form in this case, in $L^2(\Omega,\Lambda^2(\RR^3))$.

\begin{defi}[Isoflux problem]
The isoflux problem relative to $\Omega $ and a vector field $B_0\in C_T^{0,1}(\Omega, \RR^3)$,  is  the question of maximizing  over $\NN$ the ratio 
\begin{equation*}
%\label{ratio}
\R(\ga):=\dfrac{\pr{B_0,\ga}}{|\ga|}.
\end{equation*}
\end{defi}
Obviously, the term ``isoflux" is an abuse of language, since this rather corresponds to an ``isoperimetric maximal flux problem."

In the next subsection, we will see how to derive this problem formally from \eqref{modgl}, while identifying the suitable vector field $B_0$ that we need to use. 

Our main independent result on the isoflux problem is the following.%, where we  let $\NN$ be the space of normal $1$-currents supported in $\overline\Omega$, with boundary supported on $\partial\Omega$.

\begin{theorem}\label{teo:ratio}
Assume $\Omega$ is a bounded open set and $B_0\in C^0(\Omega, \RR^3)$. The supremum of the ratio $\R$ over $\NN$ is achieved. Moreover, if
\begin{equation} \label{conddd}
\sup_{\mathcal C_{\mathrm{loops}}}\R<\sup_{\NN}\R,
\end{equation}
where $\mathcal C_{\mathrm{loops}}$ denotes the space of closed oriented Lipschitz curves (that is, loops) supported in $\overline \Omega$, 
then the supremum of the ratio $\R$ over $\NN$ is achieved at an element of $X$ which is not a loop, and hence has two endpoints in $\partial \Omega$.
%at a simple oriented Lipschitz curve $\gamma$ supported in $\overline\Omega$, with multiplicity $1$ seen as a $1$-current.
\end{theorem}

\begin{remark} Condition \eqref{conddd} cannot be removed,  a counter example is described in  Remark \ref{remtorus}. 
\end{remark}

\begin{remark}
If we assume that  $B_0\times\vec\nu =0$ on $\partial \Omega$, one immediately sees that if a simple Lipschitz curve $\gamma$ in $\NN$ achieves the supremum of the ratio over $\NN$, then $\gamma$ is a closed loop supported in $\Omega$ or $\gamma\cap \partial\Omega=\{b(\gamma),e(\gamma)\}$, where $b(\gamma)$ and $e(\gamma)$ denote the endpoints of the curve. In both cases, $\gamma$ cannot have a portion on $\partial \Omega$, because that would increase the denominator of the ratio, leaving the numerator unchanged. 
\end{remark}

The proof of this theorem relies on Smirnov's decomposition  of currents into the sum of Lipschitz curves and  a solenoidal (i.e., divergence-free) charge \cite{Smi}, for details on this terminology see Section \ref{secisoflux}. 

To show that the vorticity remains bounded near the first critical field, we will need   a non-degeneracy condition of the isoflux problem to be satisfied by $\Omega$ and $B_0$, which we introduce for the first time in this paper. 
 Given a $1$-current $\ga$, we let
\begin{equation*}%\label{defstar}
\|\ga\|_*\colonequals \sup_{\|B\|_{C_T^{0,1}(\Omega,\RR^3)}\leq 1}\pr{B,\ga}
\end{equation*}
be the dual norm of $C_T^{0,1}(\Omega,\RR^3)$.

\begin{condition}[Nondegeneracy]\label{nondegencond} There exists a unique curve $\ga_0$ in $X$ such that
\begin{equation}\label{cond1}
\R(\ga_0) = \sup_{\ga\in\NN}\frac{\pr{B_0,\ga}}{|\ga|}.
\end{equation}
Moreover, there exists constants $c_0,N>0$ depending on $\Omega$ and $B_0$ such that 
\begin{equation}\label{cond2}
\R(\ga_0)-\R(\ga)\geq C_0\min\left(\|\ga-\ga_0\|_*^N,1\right)
\end{equation}
for every $\ga\in X$. 
\end{condition}

The vector field $B_0$ that we will work with for the analysis of the Ginzburg--Landau functional is the one constructed in 
 \cite{Rom2}. It appears in the Hodge decomposition of $A_0$ in $\Omega$ --- where $(e^{ih_\ex\phi_0},h_\ex A_0)$ is the Meissner state, see below --- as $A_0 = \curl B_0 + \nabla \phi_0$,  supplemented with the conditions $\diver B_0 = 0$, and $B_0\times\vec\nu = 0$ on $\partial\Omega$. Moreover, it is such that
\begin{equation}\label{B}
\int_\Omega (- \Delta B_0+ B_0-H_{0,\ex})\cdot A=0,
\end{equation}
for any divergence-free $A\in C_0^\infty(\Omega,\RR^3)$. It is worth pointing out that in \cite{Rom2}, it was mistakenly stated that $B_0$ is a weak solution to $- \Delta B_0+ B_0 =   H_{0,\ex}$ in $\Omega$. More precisely, it was claimed that \eqref{B} holds for any $A\in C_0^\infty(\Omega,\RR^3)$, but there is a mistake in a Hodge decomposition used in its derivation, which leads to an incorrect conclusion.

Also, we recall that $\phi_0$ is supplemented with the conditions $\int_\Omega \phi_0=0$ and $\nabla \phi_0\cdot \vec \nu=A_0\cdot \vec \nu$ on $\partial\Omega$.
% Equivalently, it is a divergence-free vector field solving
% \begin{equation}
% \label{eqB0}
% \left\{\begin{array}{rcl} - \Delta B_0+ B_0 & =   H_{0,\ex}   & \text{in} \ \Omega\\
% B_0\times \et{\vec\nu} & = 0  & \text{on} \ \partial \Omega.\end{array}\right.\end{equation}

Our next result shows that Condition \ref{nondegencond} holds in the case of $\Omega$ being a  ball and $B_0$ being the field  defined above, when the external field $H_{0,\ex}$ is chosen to be  constant vertical. This is the only case for which an explicit formula for $B_0$ is known.

\begin{theorem}\label{cond:ball} Condition \ref{nondegencond} holds in the case $\Omega=B(0,R)$ with $H_{0, \ex}= \hat z$ and $B_0$ defined above, with $\ga_0$ being the vertical diameter seen as a $1$-current with multiplicity $1$, oriented in the direction of positive $z$ axis, and $N=2$. 
\end{theorem}

Several interesting open questions remain about the isoflux problem:
\begin{itemize}
\item Prove that  condition \eqref{conddd} is satisfied if  $B_0$ is the Meissner field  (as defined above), at least for a large class of applied fields $H_{0,\ex}$ and domains $\Omega$.

\item Prove that the nondegeneracy condition is generic in a suitable sense, and true for a large class of explicit examples.

\item Establish a complete regularity theory for maximizers, as is done for the isoperimetric problem. In this regard, let us mention that regularity results were established in \cite{MonSte} for minimizers of a $\ga$-limit functional of Gross--Pitaevski, whose minimization leads to a problem similar to the isoflux problem, except for the fact that they deal with a weighted length term, with a weight that vanishes on $\partial \Omega$, and with a vector field $B_0$ that also vanishes (to order $k$) on $\partial\Omega$. Therefore, one would need to check if these results can adapt to our situation.
\end{itemize}

%Throughout the article, we write $\ga(B)=\int_\Omega \ga\wedge B$ for any $1$-current $\ga$ and any vector field $B\in C_T^{0,1}(\Omega,\RR^3)$. 
\subsection{Formal derivation of the isoflux problem and of the first critical field}\label{secformal} %\texorpdfstring{$H_{c_1}$}{Hc\textoneinferior}}\label{secformal}
We now explain how to formally derive the isoflux problem as done for the first time  in \cite{Rom2}. Let us also emphasize that \cites{AlaBroMon,BalJerOrlSon2} already involve the  optimum in the isoflux problem. 

Let us first  introduce  further concepts and notation from the theory of currents and differential forms. 
In Euclidean spaces, vector fields can be identified with $1$-forms. In particular, a vector field $F=(F_1,F_2,F_3)$ can be identified with the $1$-form $F_1dx_1+F_2dx_2+F_3dx_3$. We
use the same notation for both the vector field and the $1$-form. 
In addition, a vector field $F$ satisfying the boundary condition $F\times \vec\nu=0$ on $\partial \Omega$ is equivalent to a $1$-form $F$ such that $F_T=0$ on $\partial \Omega$, where $F_T$ denotes the restriction of the tangential part of F to $\partial \Omega$.

We define the superconducting current of a pair $(u,A)\in H^1(\Omega,\C)\times H^1(\Omega,\RR^3)$ as the $1$-form
$$
j(u,A)=(iu,d_A u)=\sum_{k=1}^3 (iu,\partial_k u-iA_ku)dx_k
$$
and  the gauge-invariant vorticity $\mu(u,A)$ of a configuration $(u,A)$ through
$$
 \mu(u,A)=dj(u,A)+dA.
$$
Thus $\mu(u,A)$ is an exact $2$-form in $\Omega$. % acting on couples of vector fields $(X,Y)\in \RR^3\times \RR^3$ with the standard rule that $dx_i\wedge dx_j(X,Y)=X_iY_j-X_jY_i$. 
It can also be seen as a $1$-dimensional current, which is defined through its action on $1$-forms by the relation
$$
\mu(u,A)(\phi)=\int_\Omega \mu(u,A)\wedge \phi.
$$
 The vector field corresponding to $\mu(u,A)$ (i.e. the $J(u,A)$ such that $\mu(u,A)\wedge \phi = \phi(J(u,A))\,dV$ where $dV$ is the euclidean volume form, is at the same time a gauge-invariant analogue of twice the Jacobian determinant, see for instance \cite{JerSon}, and a three-dimensional analogue of the gauge-invariant vorticity of \cite{SanSerBook}.

It is worth recalling that the boundary of a $1$-current $T$ relative to a set $\Theta$ is a $0$-current $\partial T$, and that $\partial T=0$ relative to $\Theta$ if $T(d\phi)=0$ for all $0$-forms $\phi$ with compact support in $\Theta$.
In particular, an integration by parts shows that the $1$-dimensional current $\mu(u,A)$ has zero boundary relative to $\Omega$.

We let $\mathcal D^k(\Theta)$ be the space of smooth $k$-forms with compact support in $\Theta$. For a $k$-current $T$ in $\Theta$, we define its mass by
$$
|T|(\Theta)\colonequals \sup \left\{T(\phi) \ | \ \phi\in \mathcal D^k(\Theta), \ \|\phi\|_\infty\leq 1\right\}
$$ 
and by
$$
\|T\|_{\F(\Theta)}\colonequals\sup \left\{T(\phi) \ | \ \phi\in \mathcal D^k(\Theta), \ \max\{\|\phi\|_\infty,\|d\phi\|_\infty\}\leq 1\right\}
$$
its flat norm. 
\begin{remark}
For $0$-currents, the flat and $(C_0^{0,1})^*$ norms coincide, whereas for $k$-currents the former is stronger than the latter.
\end{remark}

The Ginzburg--Landau model  admits a unique state, modulo gauge transformations,  that we will call ``Meissner state'', obtained by minimizing   $GL_\ep(u,A)$ under the constraint $|u| = 1$, so that in particular it is independent of $\ep$.  In the  gauge  where $\diver A = 0$, this state is of the form
\begin{equation}
\label{refA0}(e^{ih_\ex\phi_0},h_\ex A_0),
\end{equation}  
where $\phi_0$, $A_0$ depend only on $\Omega$ and $H_{0,\ex}$, and was  first identified in \cite{Rom2}. We call it Meissner state in reference to the Meissner effect in physics, i.e. the complete repulsion of the magnetic field by the superconductor when the superconducting density  saturates at $|u|=1$. It is not a true critical point of \eqref{modgl} (or true solution of the associated Euler--Lagrange equations) but is a good  approximation of one as $\ep \to 0$. It  corresponds to a situation with perfect superconductivity and no vortices.
The energy of this state is easily seen to be proportional to $h_\ex^2$, we write
\begin{equation*}
%\label{defJ} 
GL_\ep(e^{ih_\ex\phi_0}, h_{\ex}A_0)=: h_{\ex}^2 J_0.
\end{equation*}
 
 The first critical field occurs  when other competitors with vortices  have an energy strictly less than $ h_{\ex}^2 J_0$, as first studied in \cite{Rom2}.

 We now recall the  algebraic splitting of the Ginzburg--Landau energy from \cite{Rom2}. This decomposition of the energy is important as it allows to follow the roadmap of \cite{SanSerBook} in three dimensions.
 \begin{proposition} For any sufficiently integrable $(\u,\A)$, letting $u=e^{-ih_\ex\phi_0}\u$ and $A=\A-h_\ex A_0$, where $(e^{ih_\ex\phi_0},h_\ex A_0)$ is the approximate Meissner state, we have
\begin{equation}\label{Energy-Splitting}
GL_\ep(\u,\A)=h_\ex^2 J_0+F_\ep(u,A)%+\frac12\int_{\RR^3\setminus \Omega}|\curl A|^2
-h_\ex\int_\Omega \mu(u,A)\wedge B_0+r_0,
\end{equation}
where $F_\ep(u,A)$ is as in \eqref{freee} and
$$ 
r_0=\frac{h_\ex^2}2 \int_\Omega (|u|^2-1)|\curl B_0|^2.
$$
In particular, $|r_0|\leq C\ep h_\ex^2F_\ep(|u|,0)^{\frac12}$.
\end{proposition}
This proposition thus allows, up to a small error $r_0$, to exactly separate the energy of the Meissner state $h_{\ex}^2 J_0$, the positive free energy cost $F_\ep$ and the magnetic gain $-h_{\ex} \int \mu(u,A) \wedge B_0$.

 The constructions from \cite{Rom} will allow us to replace $\mu(u, A)$ (up to a small error as $\ep \to 0$) by a current supported on Lipschitz curve $\Gamma$, see Theorem \ref{theorem:epslevel} below. In particular $\int \mu(u, A) \wedge B_0$ can be replaced by $\pr{\Gamma, B_0}$.
 On the other hand, the energy cost $F_\ep (u, A)$  of a  vortex on a curve $\Gamma$ is at least $\pi |\Gamma| \lep$, as seen  in Theorem \ref{theorem:epslevel} from \cite{Rom} below. Neglecting $\int_{\RR^3 \backslash \Omega} |\curl A|^2$, we arrive at 
 $$GL_\ep (\u, \A) - h_{\ex}^2 J_0\simeq \pi |\Gamma | \lep   - h_{\ex} \pr{\Gamma, B_0}.$$
 Thus a vortex line both costs energy proportional to its length and gains energy proportional to the flux $\pr {\Gamma, B_0}$ which it can carry through. A vortex line can be favorable if and only if 
 $h_{\ex} \ge H_{c_1}^0$ where 
\begin{equation*}%\label{defhc1}
H_{c_1}^0\colonequals \frac\lep{2\R_0},%\|B_0\|_*},
\end{equation*}
and %we recall
\begin{equation}\label{R0}
\R_0\colonequals \sup\limits_{\ga \in X}\frac{\pr{B_0,\ga}}{|\ga|} = \sup_{\Gamma \in X} \R(\Gamma).
\end{equation}
Moreover, near $H_{c_1}^0$, favorable   vortex lines are maximizers of the isoflux problem. We will denote by $\Gamma_0$ such a maximizer (given for instance by Theorem \ref{teo:ratio}). 
As $h_{\ex}$ gets further increased beyond $H_{c_1}^0$, a multiplicity of vortex lines identical to $\Gamma_0$ could be favorable, however this is when the repulsion between close vortices of same degree, not yet counted in $F_\ep $, starts to matter,  prevents more than one line from arising at $H_{c_1}^0$, at least  if the nondegeneracy Condition  \ref{nondegencond} holds. When $h_{\ex}$ is further increased, there is (as in two dimensions) a competition between the gain in energy due to multiple lines and the cost of the repulsion, which requires further analysis. A main result of our paper is, by  quantifying the energetic cost of  repulsion between lines via an adaptation of the method of \cite{SanSer2}, to show that the vorticity (or number of vortex lines) cannot explode too fast as one passes $H_{c_1}$. Further work to find the optimal increasing number of curves as $h_\ex$ passes increasing threshholds, is carried out in \cite{RomSanSer}.
 
We now finish this formal derivation by referring the reader to the statement of the three dimensional $\ep$-level estimates for the Ginzburg--Landau functional provided in \cite{Rom}, which provides both approximation of the vorticity and essentially sharp energy lower bounds in terms of these approximations; see Theorem \ref{theorem:epslevel} in Section \ref{sec4}. More precisely, the energy lower bound in this theorem  contains an (unavoidable) error of size $\log \lep$ which limits the precision of the results in the formal derivation outlined above. Improving on this error requires knowing an a priori bound on the vorticity and following a different route for proving lower bounds, which explains why we will place for simplicity further assumptions on the domain in Theorem \ref{teo:Hc1} below.

\subsection{Main result on Ginzburg--Landau}
Throughout the paper, we assume that $H_\ex\in L^2_{\loc} (\RR^3, \RR^3)$ is such that $\diver H_\ex=0$ in $\RR^3$. In particular, we deduce that there exists a vector potential $A_\ex \in H^1_{\loc}(\RR^3,\RR^3)$ such that 
$$\curl A_\ex= H_\ex\quad \text{and} \quad \diver A_\ex=0 \ \text{in} \ \RR^3.$$
The natural space for the minimization of $GL_\ep$ in 3D is 
$H^1(\Omega, \C) \times [A_\ex+ H_{\curl}]$ where 
$$H_{\curl}:= \{ A \in H^1_{\loc} (\RR^3, \RR^3) | \curl A \in L^2(\RR^3,\RR^3)\},$$ see \cite{Rom2}.

%By construction of $\nu_\ep$ (see \cite{Rom}*{Section 5.2}), we can write
%$$
%\nu_\ep=2\pi \sum_{i\in I_\ep}\ga_i^\ep,
%$$
%where the sum is understood in the sense of currents, $I_\ep$ is a finite set of indices, and $\ga_i^\ep\in X$ for every $i\in I_\ep$.

Our next result provides a detailed characterization of the vorticity of configurations whose Ginzburg--Landau energy is bounded above by the energy of the approximation of the Meissner solution, for applied fields whose strength is close to $H_{c_1}$. It shows that they have essentially a bounded number of vortex lines, all very close to $\Gamma_0$ in a quantified way as $\ep \to 0$. 

\begin{theorem}\label{teo:boundedvorticity}
Assume that Condition \ref{nondegencond} holds. Then, for any $K>0$ and $\alpha\in(0,1)$, there exists positive constants $\ep_0,C>0$ depending on $\Omega$, $B_0$, $K$, and $\alpha$ such that the following holds. 

For any $\ep<\ep_0$ and any $h_\ex<H_{c_1}^0+K\log |\log \ep|$, if $(\u,\A)$ is a configuration in $H^1(\Omega,\C)\times [A_\ex+H_{\curl}]$ such that $GL_\ep(\u,\A)\leq h_\ex^2 J_0$,
then, letting $u=e^{-ih_\ex \phi_0}\u$ and $A=\A-h_\ex A_0$, there exists ``good'' Lipschitz curves $\ga_1,\dots,\ga_{N_0}\in X$ and $\tga\in\NN$ such that $N_0\le C$ and 
\begin{enumerate}[leftmargin=*,font=\normalfont]
\item for $1\le i\le N_0$, we have 
$$
\R(\ga_0) - \R(\ga_i)\le C\frac{\log\lep}{\lep};
$$
\item for $1\le i\le N_0$, we have
$$
\|\ga_i-\ga_0\|_*\leq C\(\dfrac{\log\lep}{\lep^\alpha}\)^\frac1N \quad\mathrm{and}\quad \big||\ga_i|-|\ga_0|\big|\leq C\(\dfrac{\log\lep}{\lep^\alpha}\)^\frac1N;
$$
\item $\tga$ is a sum in the sense of currents of curves in $X$ such that
$$
|\tga|\leq C\frac{\log\lep}{\lep^{1-\alpha}};
$$
\item we have 
$$
\left\|\mu(u,A)-2\pi\sum_{i=1}^{N_0} \ga_i-2\pi \tga\right\|_{\(C_T^{0,1}(\Omega,\RR^3)\)^*}\leq C\lep^{-2}.
$$
\end{enumerate}
\end{theorem}

%with $GL_\ep(\u,\A)\leq GL_\ep(\et{e^{ih_\ex\phi_0}},h_\ex A_0)=h_\ex^2\et{J_0}$, where $(\et{e^{ih_\ex\phi_0}},h_\ex A_0)$ is the approximate Meissner state, 
%then by defining $u=\et{e^-{ih_\ex\phi_0}}\u$ and $A=\A-h_\ex A_0$ --- where $(e^{ih_\ex\phi_0}, h_\ex A_0)$ is the Meissner state --- and letting $\nu_\ep=2\pi \sum_{i\in I_\ep} \ga_i^\ep$ be the vorticity approximation of the configuration $(u,A)$ given by Theorem \ref{theorem:epslevel} (with $n$ sufficiently large), and writing
%$$
%\nu_\ep=2\pi \left(\sum_{i\in I_\good^\ep} \ga_i^\ep+\sum_{i\in I_\bad^\ep}\ga_i^\ep\right)
%$$ \com{now multiple of $2\pi$!}
%with 
%$$
%I_\good^\ep=\left\{i\in I_\ep \ \vline \  \R(\ga_0)-\R(\ga_i^\ep)\leq |\log\ep|^{-\frac14} \right\}\quad \mbox{and}\quad I_\bad^\ep=I_\ep\setminus I_\good^\ep,$$
%we have:
%\begin{enumerate}[leftmargin=*]
%\item $N_\good^\ep:=\#(I_\good^\ep)\leq C$,
%\item $\sum_{i\in I_\bad^\ep}|\ga_i|\leq C\dfrac{\log|\log\ep|}{|\log \ep|^\frac34}$, and
%\item for some constant $c>0$ (independent of $K$)
%$$
%\|\ga_i^\ep-\ga_0\|_*\leq c|\log\ep|^{-\frac1{4N}}\quad\mathrm{and}\quad \left||\ga_i|-|\ga_0|\right|\leq c|\log\ep|^{-\frac1{4N}}
%$$
%for any $i\in I_\good^\ep$.
%\end{enumerate}
%In particular, we have that $|\nu_\ep|(\Omega)\leq C$.
%\end{theorem}
The way we proceed to prove this theorem is as follows: we use the energy splitting of \eqref{Energy-Splitting} and energy comparisons to show, as in the formal derivation of Section \ref{secformal}, that when $h_\ex$ is close to $H_{c_1}$, the good vortex curves are almost maximizers of the isoflux problem. The nondegeneracy condition assumed on the isoflux problem allows to show more precisely that all good vortex curves lie within a small tubular neighborhood of the unique maximizer $\Gamma_0$. Once this is obtained, we can surround this tubular neighborhood by annuli which avoid all vortices and on which the degree of $u/|u|$ is known. The excess free energy $F_\ep(u,A)$ carried on these annuli  
can be estimated  via the method of ``lower bounds on annuli" of \cite{SanSer2}, and bounded below by $C N_0^2 \lep$ where $N_0$ is 
the number $N_0$ of these good curves (or total degree around the tubular neighborhood). Since on the other hand the energetic gain of the curves is $O(N_0 \lep)$, the energy comparison shows that $N_0$ must remain bounded if $h_{\ex}\le H_{c_1}^0 + C \log \lep$.

Once  Theorem \ref{teo:boundedvorticity} is obtained, following the same formal 
derivation of Section \ref{secformal}, one can rigorously derive the value of $H_{c_1}$. Deriving it with very good precision requires however some delicate energy estimates on the free energy $F_\ep$. Here, for proof of concept, we will treat the simplest case of a straight $\Gamma_0$ (for instance in the case of a solid of revolution) with a flatness condition at the endpoints. We will consider a more general situation in the forthcoming work \cite{RomSanSer}.

The next theorem expresses that \cite{Rom2}*{Theorem 2} holds in fact up to $H_{c_1}- K_0$.
\begin{theorem}\label{teo:Hc1}
Let us assume that $\ga_0$ is straight, that $\partial\Omega$ is negatively curved or flat in a neighbourhood of the  endpoints of $\ga_0$, and that Condition~\ref{nondegencond} holds.   Then, under the assumptions of Theorem \ref{teo:boundedvorticity},  if 
\begin{equation}\label{hypde}
\exists \de>0, \ \|\ga_0-\ga\|_*\leq \de \Rightarrow\pr{B_0,\ga_0}\geq \pr{B_0,\ga},
\end{equation}
and if $h_{\ex}  \le H_{c_1}-K_0$ for a sufficiently large constant $K_0>0$ independent of $\ep$, then
any minimizer $(\u,\A)$ of $GL_\ep$ is vortex-less in the sense that $\| 1-|\u|\|_{L^\infty(\Omega,\C) }=o(1)$ as $\ep \to 0$.
\end{theorem}
Let us remark that \eqref{hypde} does not necessarily hold under the other assumptions of the theorem (since $\ga_0$ is a nondegenerate maximizer for the ratio, which locally minimizes length). Nevertheless, this assumption is not needed if one can prove that the inequality in Proposition \ref{prop:lowerbound} holds with the term $N_0|\ga_0|$ replaced by $\sum_{i=1}^{N_0}|\ga_i|$, which will be proved in the forthcoming work \cite{RomSanSer}. Also, under the hypotheses of Proposition~\ref{prop:positivity} below, \eqref{hypde} always holds.

This theorem improves the lower bound for $H_{c_1}$ provided in \cite{Rom2}. The upper bound was already sharp up to $O(1)$, since it was proved that if $h_{\ex} \ge H_{c_1} + K_1$ for a sufficiently large $K_1$ independent of $\ep$, then minimizers do have vortices. 
Hence, we now know $H_{c_1}$ up to an error $O(1)$ instead of an error $O(\log \lep)$ in \cite{Rom2} and an error $o(\lep)$ in \cite{BalJerOrlSon2}.

It is worth pointing out that when $\Omega$ is a ball, $\partial\Omega$ is not negatively curved or flat in a neighbourhood of the  endpoints of $\ga_0$, and this theorem thus does not apply. The strategy of proof of Theorem \ref{teo:Hc1} would work in this case if we knew that the good lines in Theorem \ref{teo:boundedvorticity} are located at distance $O\(|\log\ep|^{-\frac12}\)$ from $\ga_0$. The tools to prove this will be provided in \cite{RomSanSer}. Nevertheless, we now state a crucial estimate to prove that in a large class of domains, the axis of revolution achieves the supremum of the ratio $\R$ when $H_{0,\ex}$ is constant and points in the direction of the axis of revolution, which we assume without loss of generality to be $\hat z$.
\begin{proposition}\label{prop:positivity}
Assume $H_{0,\ex}=\hat z$ and that $\Omega$ is a solid of revolution around $\hat z$. Letting $A_0$ be the vector field in the Meissner state \eqref{refA0}, we have
$$
\curlcurl A_0 \cdot \hat \theta \leq 0\quad \mathrm{in}\ \RR^3,\quad \mathrm{where }\ \hat \theta =\left(\dfrac{-y}{\sqrt{x^2+y^2}},\dfrac{x}{\sqrt{x^2+y^2}},0 \right).
$$
In particular,
\begin{equation}\label{posB0}
\curl B_0 \cdot \hat \theta\geq 0\quad \mathrm{in}\ \Omega.
\end{equation}
\end{proposition}
Let us remark that this condition is crucial to show that the axis of revolution achieves the supremum of the ratio in the case of the ball. Moreover, if $\Omega$ is a solid of revolution around $\hat z$ and $\ga_0$ denotes the portion of the axis of revolution that belongs to $\Omega$ (pointing in the direction of positive $\hat z$), \eqref{posB0} implies that $\ga_0$ is a local maximizer for the ratio if $\partial\Omega$ is negatively curved or flat in a neighborhood of the endpoints of $\ga_0$ (which ensures that $\ga_0$ locally minimizes length). Unless $\curl B_0 \cdot \hat \theta $ is very concentrated near $\partial\Omega$, $\ga_0$ should be a global maximizer for the ratio $\R$. We do not prove this result here, but strongly believe it holds.

\bigskip
\noindent
{\bf Plan of the paper}.
In Section \ref{secisoflux} we consider independently the isoflux problem and prove Theorem \ref{teo:ratio}. In Section \ref{secball} we prove Theorem \ref{cond:ball} and Proposition \ref{prop:positivity}, which can be read independently.
In Section \ref{sec4} we turn to the Ginzburg--Landau functional and prove Theorem \ref{teo:boundedvorticity}.
In Section \ref{sechc1} we prove Theorem \ref{teo:Hc1}.

\noindent
\\
{\bf Acknowledgements:}
C. R. acknowledges funding from the Chilean National Agency for Research and Development (ANID) through FONDECYT Iniciaci\'on grant 11190130. He wishes to thank the support and kind hospitality of the Courant Institute of Mathematical Sciences and the Paris-Est University, where part of this work was done. 
E.S. wishes to thank the kind hospitality of the Max Planck Institute for Mathematics in the Sciences in Leipzig, where part of this work was done.
S. S acknowledges funding from NSF grant DMS-2000205 and the Simons Investigator program. Finally, the authors wish to thank the anonymous referee for numerous useful comments which helped improve this article.

\section{The isoflux problem and proof of Theorem \ref{teo:ratio}}\label{secisoflux}
%Following \cite{Smi}, we let $T=(T_1,T_2,T_3)$ be a vector charge, that is, an $\RR^3$-valued countably additive set function \com{I am not sure what that means, maybe my ignorance} defined on the Borel $\sigma$-algebra of $\RR^3$. We endow the set of all vector charges with the norm (often called variation)
%$$
%\|T\|=\mbox{var}(T)\colonequals \sup \sum_{j=1}^3 |T(E_j)|,
%$$
%where the supremum is taken over all Borel subdivisions of $\RR^3$. We denote by $\N_1(\RR^3)$ the space of compactly supported vector charges such that $\mbox{var}(\diver T)<+\infty$, where $\diver T$ is the scalar charge defined via the formula
%$$
%\diver T(u)\colonequals -T(\nabla u),\quad \forall u\in C_c^\infty (\RR^3).
%$$
%The simplest example of a compactly supported vector charge is a simple oriented curve of finite length. 

Smirnov in \cite{Smi} provides a structure theorem for 
the space  $\NN$ of normal $1$-currents with compact support in $\RR^3$.
%, via its identification with the space $\N_1(\RR^3)$\footnote{The variation of a charge becomes the mass of the corresponding current, whereas the divergence  of a charge becomes the boundary of the corresponding current}. 
An important role is played by {\em solenoids}, that is elements of $\NN$ with zero boundary. They may be identified to  vector valued finite measures.

The simplest example of current in $\NN$ is, given  an oriented Lipschitz curve $\ga:[a,b]\to\RR^3$, the current $T_\gamma$ defined by 
$$T_\gamma(\varphi) = \int_a^b \varphi(\gamma(t))\cdot \gamma'(t) \,dt,$$
for any smooth $1$-form $\varphi$. Its boundary is the $0$-current, a measure in this case, 
$$\partial T_\gamma = \delta_{\gamma(b)} - \delta_{\gamma(a)}.$$

Following \cite{Smi}, we denote by $\|T\|$ the measure defined by 
$$\|T\|(E) \colonequals \sup \sum_j |T(E_j)|,$$
where the supremum is taken over all Borel subdivisions of $E$. The number $\mbox{var}(T)$ is simply $\|T\|(\RR^3)$, which is also the mass of the current $T$.

By \cite{Smi}*{Theorem C}, for any $T\in \NN$, there exists $P,Q\in\NN$ such that 
\begin{equation}\label{smiC}T = P+Q, \quad \|T\| = \|P\|+\|Q\|,\quad \partial P = 0.\end{equation}
Moreover, there exists a borel measure $\mu$ on the set of simple oriented lipschitz curves such that 
\begin{equation}\label{smiCbis}Q = \int T_\gamma\,d\mu(\gamma),\quad \|Q\| = \int \|T_\gamma\|\,d\mu(\gamma), \quad \partial Q = \int\(\delta_{e(\gamma)} - \delta_{b(\gamma)}\)\,d\mu(\gamma),\end{equation}
where $e(\gamma)$ and $b(\gamma)$ denote the endpoints of $\gamma$. 

%$$
%Q=\int_{\mathcal C}\gamma d\mu(\gamma)
%$$ 
%with 
%$$
%\|Q\|=\int_{\mathcal C}|\gamma| d\mu(\gamma)\quad\mbox{and}\quad  \|\div Q\|=\int_{\mathcal C}|\delta_{b(\gamma)}-\delta_{e(\gamma)}|\, d\mu(\gamma)
%$$

%By \cite{Smi}*{Theorem C}, any $T\in \N_1(\RR^3)$ can be decomposed as the sum of a solenoid $P$ and a vector charge $Q$ that can be completely decomposed into simple oriented \com{closed like loops??} Lipschitz curves of finite length, that is,
%$$
%Q=\int_{\mathcal C}\gamma d\mu(\gamma)
%$$ 
%with 
%$$
%\|Q\|=\int_{\mathcal C}|\gamma| d\mu(\gamma)\quad\mbox{and}\quad  \|\div Q\|=\int_{\mathcal C}|\delta_{b(\gamma)}-\delta_{e(\gamma)}|\, d\mu(\gamma)
%$$
%where $\mu$ is a finite Borel measure on the space $\mathcal C$ of simple oriented Lipschitz curves of finite length and where given a curve $\gamma \in \mathcal C$, $b(\gamma)$ and $e(\gamma)$ denote its endpoints. 
The solenoidal part $P$ which appears in the decomposition \eqref{smiC} may itself be further decomposed. Indeed, according to  \cite{Smi}*{Theorem B}, any solenoid $P$ can be completely decomposed into elementary solenoids. We say that $T$ is an elementary solenoid if there exists a Lipschitz function $f:\RR\to \RR^3$ such that $\mbox{Lip}(f)\leq 1$, $f(\RR)\subseteq \mbox{Supp}(T)$, $\|T\|=1$, and that for every $\varphi\in C_c^\infty(\RR^3)$ the mean
$$
T(\varphi)\colonequals \lim_{s\to +\infty} \frac1{2s}\int_{-s}^s \langle f'(t),\varphi(f(t))\rangle dt
$$
exists. Thus, an elementary solenoid is so to speak an infinite Lipschitz curve averaged by its length, with no loss of mass in the limit. Note that a closed lipschitz curve is a special case, using a parametrization which loops indefinitely around the curve.  The set of elementary solenoids is denoted $\E$. Then, for any  $P\in\NN$ there exists a borel measure $\nu$ on $\E$ such that 
\begin{equation}\label{smiB} P = \int_\E \omega\,d\nu(\omega),\quad \|P\| = \int_\E \|\omega\|\,d\nu(\omega).\end{equation}

%a limit of simple Lipschitz curves with length 1, which do not lose mass in the limit. We can then write
%$$
%P=\int_{\mathcal E}\omega \, d\nu(\omega)
%$$ 
%with 
%$$
%\|P\|=\int_{\mathcal E} d\nu(\omega)\quad \mbox{and}\quad \diver P=0,
%$$
%where $\nu$ is a finite Borel measure on the space $\mathcal E$ of elementary solenoids. Let us point out that
%$$
%\|T\|=\|P\|+\|Q\|, \quad \|\diver T\|=\|\diver Q\|,
%$$ 
%and that the support of $\mu$-almost all curves $\gamma$ (resp. $\nu$-almost all elementary solenoids $\omega$) involved in the decomposition of $Q$ (resp. $P$) lie in the support of $Q$ (resp. $P$).
%

Smirnov's structure theorem summarized as \eqref{smiC}, \eqref{smiCbis}, \eqref{smiB}  will be crucial to prove Theorem \ref{teo:ratio}. Before providing a proof of this result, we state  helpful lemmas. 

\begin{lemma}\label{cvxity} Let $\mu$ be a measure on $\NN$, and let $T = \int \omega\,d\mu(\omega)$.

Then, if $\var(T) = \int \var(\omega)\,d\mu(\omega)$,  and provided the integrals exist, we have
$$ \R(T) = \frac{\int \R(\omega) \var(\omega)\,d\mu(\omega)}{\int \var(\omega)\,d\mu(\omega)}.$$
\end{lemma}
\begin{proof} The result follows immediately from the fact that for any $\omega\in\NN$, we have $\pr{B_0,\omega} = \R(\omega) \var(\omega).$
\end{proof}

\begin{lemma}
There exists a constant $C=C(\Omega)>0$ such that
$$
\sup_{\NN} \R=\sup_{T\in \NN,\var(\partial T)\leq C\var(T)}\R(T).
$$
\end{lemma}

\begin{proof} Let $\R_0$ be defined by \eqref{R0} and let $T\in\NN$ be such that $\R(T)>\R_0/2$. 
By Smirnov's structure theorem, $T=P+Q$ with $\|T\|=\|P\|+\|Q\|$ and $\|\partial T\|=\|\partial Q\|$. 

By the isoperimetric inequality (see \cite{Rom2}*{Remark 4.1}), we know that there exists $C=C(\Omega)$ such that if $|\gamma|<C^{-1}$ then
$$
\R(\gamma)\leq \frac12 \R_0.
$$
Let 
$$Q = \int_X T_\gamma\,d\mu(\gamma)$$
be the decomposition of $Q$, let $\tilde X$ be the set of curves longer than $C^{-1}$, and let 
$$
\tilde Q=\int_{\tilde X}\gamma d\mu (\gamma),\quad \tilde T = P+\tilde Q,\quad S = T - \tilde T = 
\int_{X\setminus \tilde X} T_\gamma\,d\mu(\gamma).
$$
%where $\tilde {\mathcal C}=\{\gamma \in \mathcal C \ | \ |\gamma|\geq C^{-1}\}$, and defining $\tilde T=P+\tilde Q$, by convexity of the ratio $\R$, we have
Then, using the previous lemma, 
$$\R(T) = \frac{\var(\tilde T)}{\var(T)}\R(\tilde T) + \frac{\var(S)}{\var(T)}\R(S) .$$
Since $\var(T) = \var(\tilde T) + \var(S)$, and since from the previous lemma $\R(S)\le \R_0/2 < \R(T)$, we deduce that 
\begin{equation}\label{ratios}
\R(T)\leq \R(\tilde T).
\end{equation}
But 
$$\var(\partial\tilde T) = \var(\partial\tilde Q) = \int_{\tilde X} \var(\partial T_\gamma)\,d\mu(\gamma).$$
It follows that 
$$\var(\partial\tilde T) = 2\mu(\tilde X)\le 2C \int_{\tilde X}\var(T_\gamma)\,d\mu(\gamma) \le 2C\var(\tilde T).$$
Therefore, in view of \eqref{ratios},  in computing $\R_0$ one may consider only normal currents $\tilde T$ such that 
$\var(\partial\tilde T) \le 2C\var(\tilde T).$
\end{proof}

%Moreover, for any $\gamma \in \mathcal C$, we have $\|\partial \gamma\|\leq \tilde C|\gamma|$, where $\tilde C=\tilde C(\Omega)>0$. \com{what is the link with the previous $C$? where is it used that $|\gamma|\ge C^{-1}?$} Thus
%$$\|\partial \tilde Q\|\leq \tilde C\|Q\|,$$
%which combined with \eqref{ratios} provides the result.
%\end{proof}

\begin{proof}[Proof of Theorem \ref{teo:ratio}]
Let us first prove that the supremum of the ratio over $\NN$ is achieved. For that, we let $T_n\in \NN$ be a maximizing sequence. By homogeneity of the ratio and the previous lemma, we can assume without loss of generality that $\|T_n\|=1$ and $\|\partial T_n\|\leq C(\Omega)\|T_n\|=C(\Omega)$. By the compactness of normal currents, we deduce that (up to passing to a subsequence) $T_n\to T\in \NN$ weakly in the sense of measures. Since
$$
\|T\|\leq \liminf_n\|T_n\|=1
$$
and
$$
\langle T,B_0\rangle = \lim_n \langle T_n,B_0\rangle=\lim_n \R(T_n),
$$
we deduce that
$$
\R(T)\geq \lim_n \R(T_n)=\sup_{\NN} \R.
$$

\medskip
Let us now prove the second part of Theorem~\ref{teo:ratio}. For that, we let $T\in \NN$ be such that
$$
\R(T)=\max_\NN \R.
$$
By Smirnov's structure theorem, $T=P+Q$ with $\|T\|=\|P\|+\|Q\|$ and $\|\partial T\|=\|\partial Q\|$. Observe that to prove the result it is enough to show that $P=0$, which implies that $\nu$-almost every curve in the decomposition of $Q$ must be a maximizer for the ratio. From $\|\partial T\|=\|\partial Q\|$ we know that  implies that almost every curve in the decomposition of $Q$ has endpoints on the boundary. Let $\omega$ be an elementary solenoid. We claim that
$$
\R(\omega)\leq \sup_{\mathcal C_{\mathrm{loops}}}\R,
$$
which in turns, by Lemma~\ref{cvxity}  implies that necessarily $P=0$, for otherwise $T$ cannot be a maximizer for the ratio in $\NN$. This proves the second part of Theorem~\ref{teo:ratio}.

To prove the claim we observe that, by definition of an elementary solenoid, there exists a simple Lipschitz curve $\gamma$ parametrized by arclength, such that
$$
\omega=\lim_{s\to +\infty} \frac1{2s}\gamma|_{[-s,s]}.
$$
For any $s>0$ we close $\gamma|_{[-s,s]}$ by a segment joining the endpoints and call the resulting loop $\tilde\gamma_s$. The contribution of the closing segment to $\pr{B_0,\tilde\gamma_s}$ is bounded by a constant depending only on $\Omega$, therefore 
$$\lim_{s\to +\infty}\frac1{2s}\pr{B_0,\tilde\gamma_s} = \lim_{s\to +\infty} \frac1{2s}\pr{B_0,\gamma_s} = \pr{B_0,\omega}.$$
The length of the segment being bounded as well, we have
$$\lim_{s\to +\infty}\frac{|\tilde\gamma_s|}{2s}= \lim_{s\to +\infty}\frac{|\gamma_s|}{2s} = 1 = \var(\omega).$$
It follows that 
$$\R(\omega) = \lim_{s\to +\infty} \R(\tilde\gamma_s/2s)\le \sup_{\mathcal C_{\mathrm{loops}}}\R.$$
\end{proof}

%Thus, given $\ep>0$, there exists $s_0\gg 1$ such that
%$$
%\left\| \omega - \frac1{2s_0}\gamma|_{[-s_0,s_0]} \right\|_{C_c^\infty(\RR^3)}<\ep.
%$$
%\com{why? not clear to me}
%Moreover, letting $s_0$ be larger if necessary, we can close {\color{red} into a loop} \com{?} $\gamma|_{[-s_0,s_0]}$ in such a way that 
%$$
%\left\| \frac1{2s_0}\gamma|_{[-s_0,s_0]}-\frac{\tilde\gamma}{|\tilde\gamma |} \right\|_{C_c^\infty(\RR^3)}<\ep,
%$$
%where $\tilde \gamma$ denotes the closure of $\gamma|_{[-s_0,s_0]}$.
%
%Hence, letting $C>0$ denote a constant that depends only on $\Omega$, we have
%\begin{align*}
%\R(\omega)\leq \R\left( \frac1{2s_0}\gamma|_{[-s_0,s_0]}\right)+C\ep
%&\leq \R\left(\frac{\tilde \gamma}{|\tilde\gamma|}\right)+C\ep\\
%&\leq \sup_{\mathcal C_{\mathrm{loops}}} \R +C\ep.
%\end{align*}
%Since $\ep$ is arbitrary, we conclude the proof of the claim, and the proposition is thus proved. 
%\end{proof}
%
\begin{remark}\label{remtorus} The following counterexample shows that the supremum of $\R$ may not be achieved by a simple curve if the assumption \eqref{conddd} is dropped.

%One can construct a vector field $B_0$ such that the supremum of the ratio $\R$ over $\NN$ is not achieved among Lipschitz curves. 

Consider  a torus $\tau$ compactly contained in $\Omega$. On $\tau$, let $B_0$ be a tangent vector field  with integral curve dense in $\tau$. Outside of the torus, $B_0$ can be smoothly extended so that $|B_0(x)|<1$ in $\Omega\setminus \tau$. Then, by construction, 
$$
\sup_{\NN} \R =1
$$
and is not achieved by a simple Lipschitz curve, but rather by the elementary solenoid associated to the integral curve of $B_0$.
\end{remark}

\section{Non-degeneracy condition in the case of the ball}\label{secball}
We begin this section by providing a proof for Theorem \ref{cond:ball}.
\begin{proof}
In \cite{Rom2}*{Proposition 4.1}, it was proved that, in the case of the ball, the diameter $\ga_0$ is the unique curve in $X$ satisfying \eqref{cond1}. It remains to prove \eqref{cond2} for every $\ga\in X$ with $\pr{B_0,\ga}>0$ (otherwise the result is trivial).

\medskip\noindent 
{\bf Step 1. Dimension reduction.} 

\smallskip\noindent 
Let $(r,\theta,\phi)$ denote a system of spherical coordinates, where $r$ is the Euclidean distance from the origin, $\theta$ is the azimuthal angle, and $\phi$ is the polar angle , with $\hat r, \hat \theta, \hat \phi$ the associated unit vectors. An explicit computation (see for instance \cite{Lon}) shows that
\begin{multline*}%\label{ballfield}
B_0=C(R)\hat z-\frac{3R}{r^2\sinh(R)}\left(\cosh r-\frac{\sinh r}r\right)\cos\phi \, \hat r\\ +\frac{3R}{2r^2\sinh(R)}\left(\frac{1+r^2}r \sinh r -\cosh r\right)\sin \phi \, \hat{\phi},
\end{multline*}
where $C(R)=\dfrac{3}{2R\sinh(R)}\left(\dfrac{1+R^2}R \sinh R-\cosh R\right)$.

Let $\ga\in X$ with $\pr{B_0,\ga}>0$, and choose $\theta_0$ to be some angle such that $(r,\theta_0,\phi)$ belongs to the support of $\ga$, for some $r$ and $\phi$. We also let $(x,y,z)$ denote a system of Cartesian coordinates centered at the origin and such that the plane $\{y=0\}$ coincides with the plane $\{\theta=\theta_0\}$.

Let us define 
$$
B(0,R)^{2\D}\colonequals \{(x,0,z) \in \RR^3 \ | \ x^2+z^2<R^2\}
$$
and
\begin{equation}\label{halfball}
B(0,R)^{2\D,+}\colonequals \{(x,0,z) \in \RR^3 \ | \ x^2+z^2<R^2,\ x\geq 0\}.
\end{equation}
We now project $\ga$ along the azimuthal angle onto $B(0,R)^{2\D,+}$. To do this, we consider the map $q:B(0,R)\subset \RR^3\to B(0,R)^{2\D,+}$ defined by
$$
q(r,\theta,\phi)=(r\sin\phi,0,r\cos\phi),
$$
and let 
$$
\ga_{2\D}\colonequals q \circ \ga.
$$
It is easy to check that
$$
\partial \ga_{2\D}=0\ \mbox{relative to } B(0,R)^{2\D},\quad 
\pr{B_0,\ga_{2\D}}=\pr{B_0,\ga},\quad \mbox{and} \quad |\ga_{2\D}|\leq |\ga|.
$$
We observe that $\ga_0=(\ga_0)_{2\D}$.

Since $\pr{B_0,\ga}>0$, one necessarily has $|\ga_{2\D}|>0$. Let us observe that
\begin{align}
\R(\ga_0)-\R(\ga)&=\R(\ga_0)-\R(\ga_{2\D})+\R(\ga_{2\D})-\R(\ga)\notag\\
&=\R(\ga_0)-\R(\ga_{2\D})+
\frac{\R(\ga_{2\D})}{|\ga|}(|\ga|-|\ga_{2\D}|).\label{reduc1}
\end{align}

The isoperimetric inequality allows one to show that if $|\ga_{2\D}|\ll 1$ then the ratio $\R(\ga_{2\D})$ is small; see \cite{Rom2}*{Remark 4.1}. Also note that, if $\R(\ga_{2\D})\ll1$, from \eqref{reduc1} we deduce that
$$
\R(\ga_0)-\R(\ga)\geq \R(\ga_0)-\R(\ga_{2\D})\geq \frac12\R(\ga_0). 
$$
We can therefore assume from now on that $\R(\ga_{2\D})\geq c$ and $|\ga_{2\D}|\geq c$ for some fixed constant $c>0$. In addition, as we shall show in Step 2, we can assume without loss of generality that $|\ga|$ is bounded above, provided $c_0$ is chosen sufficiently small in \eqref{cond2}. For these reasons, hereafter we assume that 
$$
c_1\leq |\ga_{2\D}| \leq |\ga|\leq c_2
$$
for some constants $c_1,c_2>0$ depending on $B_0$ and $R$ only. 

\medskip 
We claim that there exists $C>0$ (depending only on $R$) such that
\begin{equation}
\|\ga-\ga_{2\D}\|_*\leq C\sqrt{|\ga|-|\ga_{2\D}|} \label{reduc2}.
\end{equation}
Indeed, let us consider $B\in C_T^{0,1}(\Omega,\RR^3)$ and observe that
$$
\pr{B,\ga}-\pr{B,\ga_{2\D}}=\int \left(B(\ga)-B(\ga_{2\D})\right)\cdot (\ga_r'\hat r+\ga_z'\hat z)+\int B(\ga)\cdot (\ga_\theta'\hat \theta).
$$Here we are using cylindrical coordinates.
By our choice of $\theta_0$ and definition of $\ga_{2\D}$, one can immediately check that
$$
\left|B(\ga)-B(\ga_{2\D})\right|\leq \|B\|_{C_T^{0,1}(\Omega,\RR^3)}(|\ga|-|\ga_{2\D}|),
$$
which implies that
$$
\left|\int \left(B(\ga)-B(\ga_{2\D})\right)\cdot (\ga_r'\hat r+\ga_z'\hat z) \right|\leq \|B\|_{C_T^{0,1}(\Omega,\RR^3)}(|\ga|-|\ga_{2\D}|)|\ga_{2\D}|.
$$
On the other hand,
$$
\left| \int B(\ga)\cdot (\ga_\theta'\hat \theta) \right|\leq \|B\|_{L^\infty(\Omega,\RR^3)}\sqrt{|\ga|-|\ga_{2\D}|}.
$$
Combining the two, we obtain
$$
\|\ga-\ga_{2\D}\|_*\leq \left(1+\sqrt{|\ga|-|\ga_{2\D}|}\right)\sqrt{|\ga|-|\ga_{2\D}|}\leq C \sqrt{|\ga|-|\ga_{2\D}|},
$$
and the claim is thus proved. 

\medskip
By inserting \eqref{reduc2} and
$$
\|\ga-\ga_0\|_*\leq \|\ga-\ga_{2\D}\|_*+\|\ga_{2\D}-\ga_0\|_*
$$ 
into \eqref{reduc1}, we deduce that proving \eqref{cond2} with $N=2$ reduces to showing that
\begin{equation}\label{estGamma2D}
\R(\ga_0)-\R(\ga_{2\D})\geq C_0\min\left(\|\ga_{2\D}-\ga_0\|_*^2,1\right).
\end{equation}

\noindent\medskip
{\bf Step 2. Boundedness of $|\ga|$ and proof of \eqref{estGamma2D}.} 

\noindent\smallskip
Since $\ga\in X$, one deduces that $\ga_{2\D}$ can be decomposed as
$$
\ga_{2\D}=\ga_s+\sum_{i\in I}\ga_i,
$$
where the sum is understood in the sense of currents, $I$ is a countable set of indices (not containing $0$), $\ga_s\in X$ is a curve contained in  $\overline{B(0,R)^{2\D,+}}$ having two different endpoints on $\partial B(0,R)^{2\D}$, and $\ga_i\in X$ is a loop contained in $\overline{B(0,R)^{2\D,+}}$ for every $i\in I$. Note that $\ga_s=(\ga_s)_{2\D}$ and $\ga_i=(\ga_i)_{2\D}$ for every $i\in I$.

As we shall show in Step 3, we have
\begin{equation}\label{estGammainotloop}
\R(\ga_0)-\R(\ga_s)\geq C\|\ga_0-\ga_s\|_*^2.
\end{equation}
In addition, in Step 4 we will prove that
\begin{equation}\label{estGammailoop}
\R(\ga_0)-\R(\ga_i)\geq C>0
\end{equation}
for every $i\in I$. With these two estimates at hand, let us now prove \eqref{estGamma2D}. We define 
$$
\beta_s=\frac{|\ga_s|}{|\ga_{2\D}|}\quad \mbox{and}\quad \beta_i=\frac{|\ga_i|}{|\ga_{2\D}|}\ \mbox{for every }i\in I.
$$
Observe that $\beta_s +\sum_{i\in I}\beta_i=1$ and
\begin{equation}\label{finalest1}
\R(\ga_0)-\R(\ga_{2\D})=\beta_s\left( \R(\ga_0)-\R(\ga_s)\right) + \sum_{i\in I}\beta_i \left( \R(\ga_0)-\R(\ga_i)\right).
\end{equation}
Inserting \eqref{estGammainotloop} and \eqref{estGammailoop} into \eqref{finalest1}, we obtain
\begin{equation}
\R(\ga_0)-\R(\ga_{2\D})\geq C \beta_s \|\ga_0-\ga_s\|_*^2+C\sum_{i\in I}\beta_i.\label{finalest2}
\end{equation}
Note that the left-hand side of \eqref{finalest2} is bounded above by a constant $c_0>0$ to be fixed (otherwise, the result immediately follows). In particular, we deduce that
$$
\sum_{i\in I}\beta_i=1-\beta_s\leq C^{-1}c_0,
$$
that is $1-C^{-1}c_0\leq \beta_s$. On the other hand, by \cite{Rom2}*{Proposition 4.1} we know that $\pr{B_0,\ga_s}\leq \pr{B_0,\ga_0}$, which combined with \eqref{finalest1} lead us to
$$
c_0\geq \beta_s \left( \frac{\pr{B_0,\ga_0}}{|\ga_0|}-\frac{\pr{B_0,\ga_0}}{|\ga_s|}\right)
$$
that is
$$
|\ga_s|\leq \frac{\beta_s \pr{B_0,\ga_0}}{\beta_s\R(\ga_0)-c_0}.
$$
Therefore
$$
|\ga_{2\D}|=\beta_s^{-1}|\ga_s|\leq \frac{\pr{B_0,\ga_0}}{\beta_s\R(\ga_0)-c_0},
$$
which, provided $c_0$ is small enough, yields that $|\ga_{2\D}|$ is bounded by a constant depending only on $B_0$ and $R$. In turn, this implies that $\pr{B_0,\ga_{2\D}}$ is also bounded. Besides,
$$
c_0\geq \R(\ga_0)-\R(\ga_{2\D})\geq \R(\ga_0)-\R(\ga),
$$
that is, recalling that $\pr{B_0,\ga_{2\D}}=\pr{B_0,\ga}$,
$$
|\ga|\leq \frac{\pr{B_0,\ga_{2\D}}}{\R(\ga_0)-c_0}.
$$
We thus obtain the bound on $|\ga|$ claimed in the previous step. 

\medskip
Finally, by noting that 
$$
\|\ga_0-\ga_{2\D}\|_*\leq \|\ga_0-\ga_s\|_*+\sum_{i\in I}\|\ga_i\|_*,
$$
and since by the isoperimetric inequality we have $\|\ga_i\|_*\leq C|\ga_i|^2$ for every $i\in I$, we deduce that
$$
\|\ga_0-\ga_{2\D}\|_*\leq \|\ga_0-\ga_s\|_*+C\sum_{i\in I}\beta_i^2.
$$
Inserting into \eqref{finalest2}, we are led to
$$
\R(\ga_0)-\R(\ga_{2\D})\geq C_0\|\ga_0-\ga_{2\D}\|_*^2,
$$
which concludes the proof of \eqref{estGamma2D}.

\medskip\noindent 
{\bf Step 3. Proof of \eqref{estGammainotloop}.} 

\smallskip\noindent
We define $S_{\ga_s}$ as the subset of $B(0,R)^{2\D,+}$ enclosed by the loop formed by the union of $\ga_s$ and the curve lying on $\partial B(0,R)\cap \partial B(0,R)^{2\D,+}$ which connects the endpoints of $\ga_s$ oriented according to the orientation of $\ga_s$. Since $B_0\times \vec\nu =0$ on $\partial B(0,R)$, the Stokes' theorem yields
\begin{equation*}%\label{stokes1}
\pr{B_0,\ga_s}=\int_{S_{\ga_s}} \curl B_0\cdot \hat y.
\end{equation*}
In particular, we deduce that
$$
\pr{B_0,\ga_s}\leq \|\curl B_0\|_{\infty}\mbox{Area}(S_{\ga_s})\leq C,
$$
where $C$ is a constant that depends on $B_0$ and $R$ only.

\medskip
We define $S\colonequals S_{\ga_0}\setminus S_{\ga_s}$, where we remark that $S_{\ga_0}=B(0,R)^{2\D,+}$, and observe that 
$$
\pr{B_0,\ga_0}-\pr{B_0,\ga_s}=\int_S \curl B_0\cdot \hat y.
$$

\medskip\noindent
{\bf Step 3.1.} We claim that
\begin{equation}\label{estnormcurl} 
\pr{B_0,\ga_0}-\pr{B_0,\ga_s}\geq C_0(R)\|\ga_0-\ga_s\|_*^2,
\end{equation}
where $C_0(R)$ is a constant that depends on $R$ only. 

To prove this, we first point out that an explicit computation gives
\begin{equation*}%\label{positivecurl}
\curl B_0\cdot \hat y=\frac{3R}{2\sinh R}\left(\cosh r-\frac{\sinh r}r \right)\frac{\sin \phi}r\geq 0\quad \mathrm{in}\ B(0,R)^{2\D,+}.
\end{equation*}
By noting that, for any $r>0$,
$$
\cosh r-\frac{\sinh r}r\geq \frac{r^2}3
$$
and using Cartesian coordinates, we find
\begin{equation}\label{ineq1}
\pr{B_0,\ga_0}-\pr{B_0,\ga_s}\geq \frac{R}{2\sinh R} \int_S x dx dz.
\end{equation}

On the other hand, we observe that 
$$
\|\ga_0-\ga_s\|_*\leq \int_S dx dz,
$$
which directly follows from the definition of the star-norm and Stokes' theorem.
We define $S_z\colonequals \{x \ | \ (x,z)\in S \}$ and note that
$$
\int_S dxdz=\int_{-R}^R|S_z|dz.
$$
By monotonicity of $x$, we have
\begin{equation}\label{ineq2}
\int_{S_z} x dx \geq \int_0^{|S_z|}xdx=\frac12 |S_z|^2.
\end{equation}
But, from the Cauchy-Schwarz inequality, we find
$$
\sqrt{2R}\left( \int_{-R}^R |S_z|^2dz\right)^\frac12\geq \int_{-R}^R |S_z|dz.
$$
Combining with \eqref{ineq1} and \eqref{ineq2}, we are led to \eqref{estnormcurl}

\medskip\noindent
{\bf Step 3.2.} We now use \eqref{estnormcurl} to prove \eqref{estGammainotloop}. We consider two cases.

\medskip
First, we consider the case $|\ga_s|\geq |\ga_0|$. Observe that
$$
\R(\ga_0)-\R(\ga_s)=\frac{\pr{B_0,\ga_0}-\pr{B_0,\ga_s}}{|\ga_0|}+\frac{\pr{B_0,\ga_s}}{|\ga_0||\ga_s|}\left(|\ga_s|-|\ga_0|\right)\geq \frac{\pr{B_0,\ga_0}-\pr{B_0,\ga_s}}{|\ga_0|},
$$
which combined with \eqref{estnormcurl} gives \eqref{estGammainotloop}.

\medskip
Second, we consider the case $|\ga_s|<|\ga_0|$. Once again, by \cite{Rom2}*{Remark 4.1}, we can assume that $|\ga_s|\geq c>0$ for some constant $c$ depending on $B_0$ only.

For $a,b\in[0,\pi]$ with $a\leq b$, let us define
$$
S_{a,b}\colonequals \{(r,\phi)\ | \ 0\leq r\leq R,\ a\leq \phi \leq b\}.
$$ 
We let $\phi_1$ (resp. $\phi_2$) be the maximum angle (resp. minimum angle) for which $S_\ga \subseteq S_{\phi_1,\phi_2}$ and define the curve 
$$
\ga_{\phi_1,\phi_2}=\partial S_{\phi_1,\phi_2}\setminus \{(R,\phi) \ | \ \phi_1<\phi<\phi_2\}
$$
oriented from $(R,\phi_2)$ to $(R,\phi_1)$. We also let $L_{\phi_1,\phi_2}$ be the straight segment connecting $(R,\phi_2)$ to $(R,\phi_1)$. This is depicted in Figure \ref{figure1}. We note that, since $|\ga_s|<|\ga_0|$, $\phi_1+\phi_2>0$.

\begin{figure}
\centering
\includegraphics[]{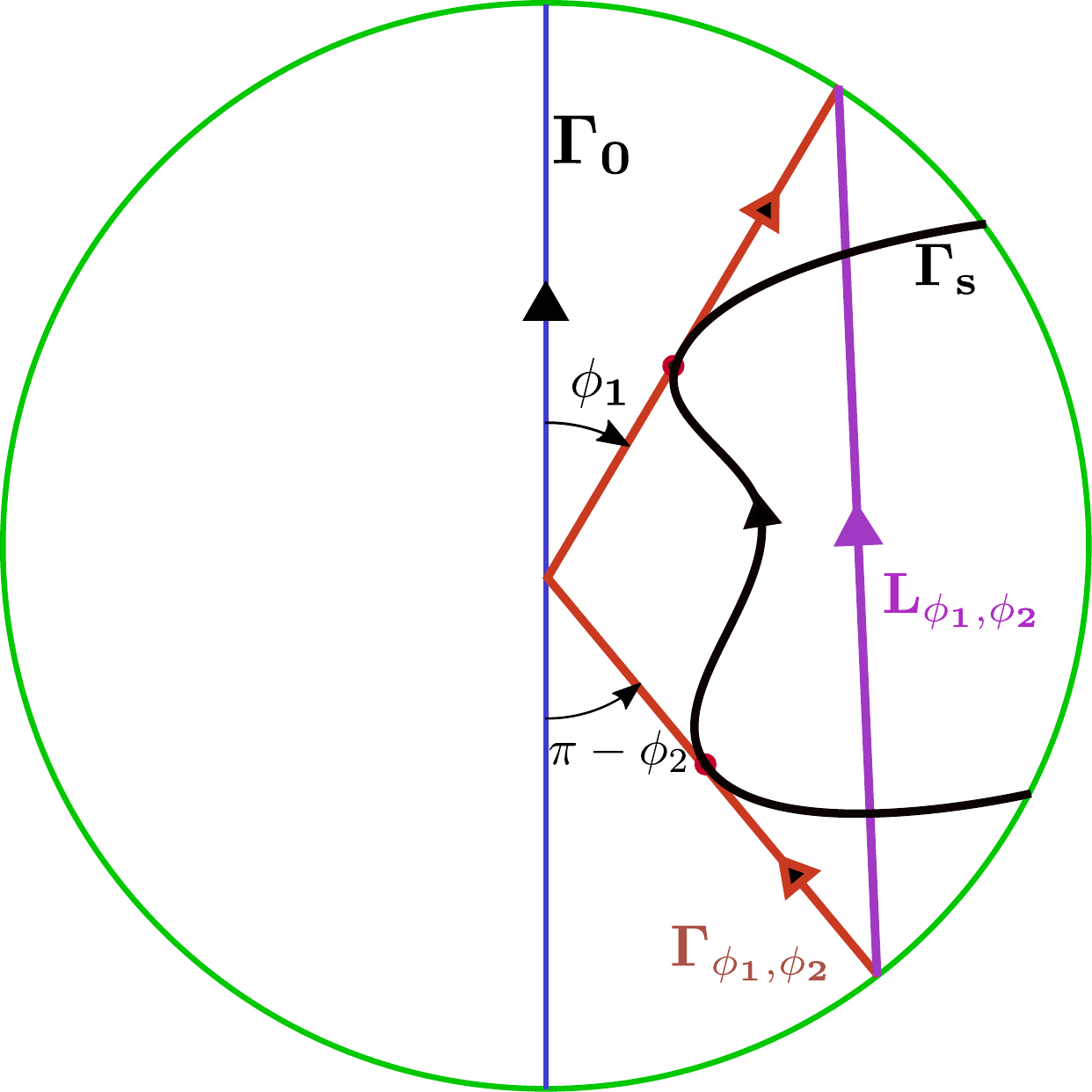}
\caption{Depiction of $\ga_{\phi_1,\phi_2}$, $L_{\phi_1,\phi_2}$, $\ga_0$ and $\ga_s$. }
\label{figure1}
\end{figure}

Let us recall from \cite{Rom2}*{Section 4} that $\pr{B_0,\ga_{\phi_1,\phi_2}}\geq \pr{B_0,\ga}$, $|\ga|\geq  |L_{\phi_1,\phi_2}|$, and
\begin{equation}\label{413}
\R(\ga_s)\leq \frac{\pr{B_0,\ga_{\phi_1,\phi_2}}}{|L_{\phi_1, \phi_2} |}=\cos\left(\frac{\phi_1-(\pi-\phi_2)}2\right) \R(\ga_0)=\sin\left(\frac{\phi_1+\phi_2}2\right) \R(\ga_0)\leq \R(\ga_0).
\end{equation}
In particular, 
\begin{equation*}%\label{phis}
\alpha\geq \left(1-\sin\left(\frac{\phi_1+\phi_2}2\right)\right) \R(\ga_0),
\end{equation*}
where hereafter we denote $\alpha\colonequals \R(\ga_0)-\R(\ga_s)$. Observe that from \eqref{413}
\begin{multline}\label{alpha1}
\alpha=\R(\ga_0)-\frac{\pr{B_0,\ga_{\phi_1,\phi_2}}}{|\ga_s|} +\frac{\pr{B_0,\ga_{\phi_1,\phi_2}}}{|\ga_s|}-\frac{\pr{B_0,\ga_s}}{|\ga_s|}\\
\geq 
\frac{\pr{B_0,\ga_{\phi_1,\phi_2}}}{|\ga_s||L_{\phi_1,\phi_2}|}\left(|\ga_s|-|L_{\phi_1,\phi_2}| \right)+
\frac1{|\ga_s|}\left(\pr{B_0,\ga_{\phi_1,\phi_2}}-\pr{B_0,\ga_s}\right).
\end{multline}
We define $S_0$ as the triangle enclosed by $\ga_{\phi_1,\phi_2}$ and $L_{\phi_1,\phi_2}$.

\medskip\noindent
{\bf Step 3.2.1. Reduction to the case $\ga_s\subset S_0$.} 

\smallskip\noindent
If $\ga_s\not\subset S_0$, we claim that there exists a curve $\tilde \ga_s \subset S_0$, whose endpoints are $(R,\phi_1)$ and $(R,\phi_2)$, such that $|\tilde \ga_s|<|\ga_s|$ and $S_{\ga_s}\subset S_{\tilde \ga_s}$. To prove this, let us first observe that from the definition of $\phi_1,\phi_2$, we know that $\ga_s$ intersects the rays $\mathcal R_1=\{(r,\phi_1) \ | \ 0\leq r\leq R\}$ and $\mathcal R_2= \{(r,\phi_2) \ | \ 0\leq r\leq R\}$. We let $r_1$ (resp. $r_2$) denote the maximum radius for which $\ga_s$ intersects $\mathcal R_1$ (resp. $\mathcal R_2$). We define $\hat \ga_s$ as the curve that is obtained by replacing the pieces of $\ga_s$ that connect $(r_1,\phi_1)$ and $(r_2,\phi_2)$ to its endpoints by the pieces of the rays $\mathcal R_1$ and $\mathcal R_2$ that connect $(r_1,\phi_1)$ to $(R,\phi_1)$ and $(r_2,\phi_2)$ to $(R,\phi_2)$, oriented accordingly to the orientation of $\ga_s$. This is depicted in Figure \ref{figure2}. One immediately sees that $|\hat \ga_s|\leq |\ga_s|$ and that $S_{\ga_s}\subseteq S_{\hat\ga_s}$.

We now define $\tilde \ga_s$ as the curve that one obtains by joining the pieces of $\hat \ga_s$ that are contained in $S_0$ (we note that at least its endpoints belong to $S_0$) with the pieces of $L_{\phi_1,\phi_2}$ which are not contained in $S_{\ga_s}$, preserving the orientation of $\ga_s$. It is easy to see that $|\tilde \ga_s|<|\hat \ga_s|$ and $S_{\hat \ga_s}\subset S_{\tilde \ga_s}$, and therefore $\tilde \ga_s$ is as claimed.

\begin{figure}
\centering
\includegraphics[]{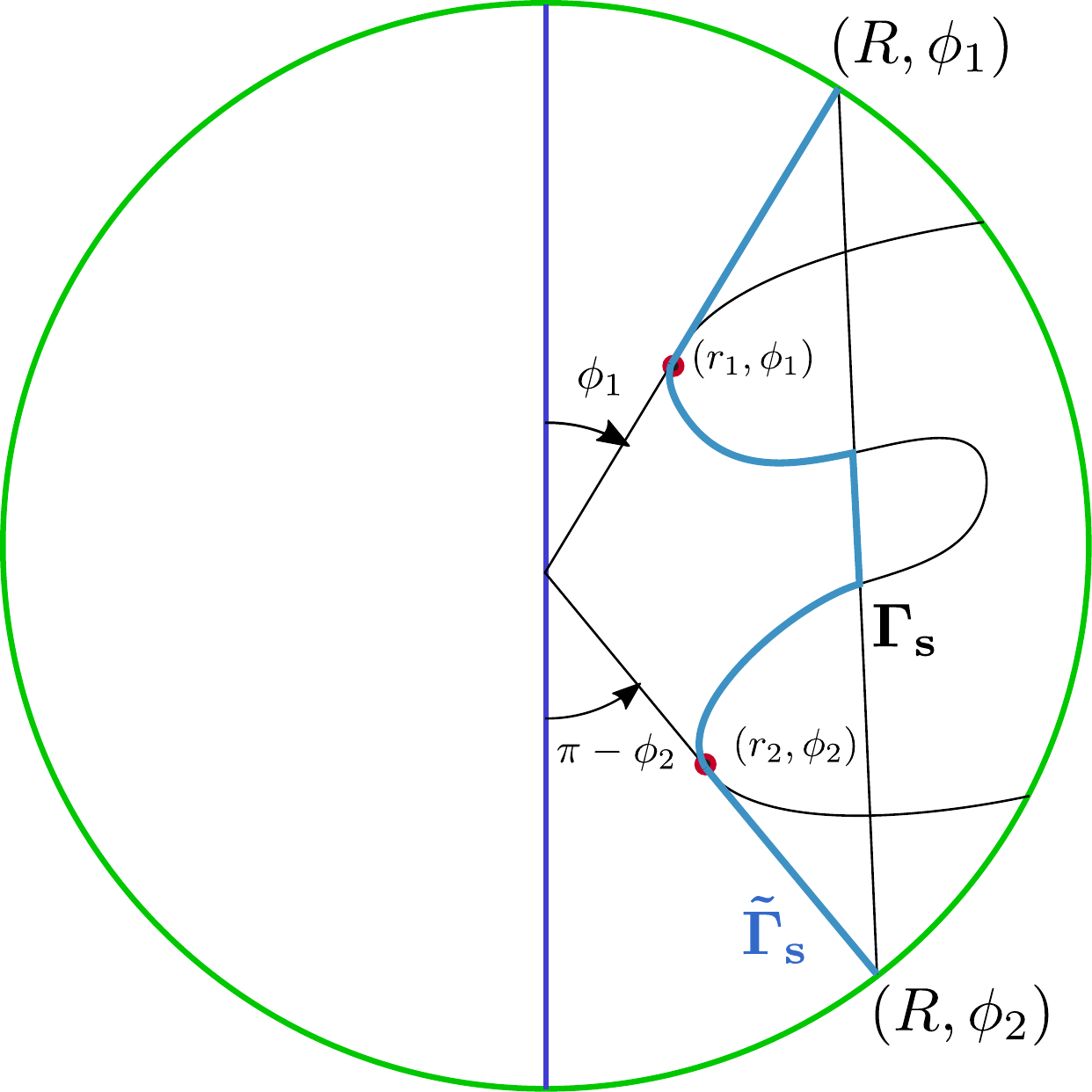}
\caption{Depiction of $\ga_s$ and $\tilde \ga_s$.}
\label{figure2}
\end{figure}

Let us now note that
$$
\alpha=\R(\ga_0)-\R(\tilde \ga_s)+\R(\tilde \ga_s)-\R(\ga_s)\geq \R(\ga_0)-\R(\tilde \ga_s)+\frac{\langle B_0,\tilde\ga_s\rangle-\pr{B_0,\ga_s}}{|\ga_s|}.
$$
Letting $\tilde S\colonequals S_{\ga_0}\setminus S_{\tilde \ga_s}$, we have 
$$
\langle B_0,\tilde\ga_s\rangle -\pr{B_0,\ga_s}=\int_{S}\curl B_0\cdot \hat y-\int_{\tilde S}\curl B_0\cdot \hat y=\int_{ S\setminus \tilde S} \curl B_0\cdot \hat y.
$$
Arguing as in  Step 3.1, we are led to
$$
\langle B_0,\tilde\ga_s \rangle -\pr{B_0,\ga_s}\geq C\|\tilde \ga_s -\ga_s\|_*^2.
$$
Hence, if one can show that
$$
\R(\ga_0)-\R(\ga)\geq C\|\ga_0-\ga\|_*^2
$$
for any curve $\ga\in S_0$, by combining this inequality with $\ga=\tilde \ga_s$ with the previous inequalities, one finds
$$
\R(\ga_0)-\R(\ga_s)\geq C\|\ga_0-\tilde \ga_s\|_*^2+C\|\tilde \ga_s- \ga_s\|_*^2\geq C\|\ga_0-\ga_s\|_*^2.
$$

\medskip\noindent
{\bf Step 3.2.2.}

\smallskip\noindent
Let us now consider the case $\ga_s\subset S_0$ and prove \eqref{estGammainotloop}.
We will consider two cases. First, given a constant $c>1$ that will be fixed afterwards, let us assume that 
\begin{equation}\label{case1}
|\ga_s|-|L_{\phi_1,\phi_2}|\geq\frac1c(|\ga_0|-|L_{\phi_1,\phi_2}|).
\end{equation}
We will prove \eqref{estGammainotloop} using the first term on the right-hand side of \eqref{alpha1}.
A straightforward geometric calculation shows that there exists a constant $C>0$ depending on $R$ such that
$$
|\ga_0|-|L_{\phi_1,\phi_2}|\geq C|S_{\ga_0}\setminus S_{L_{\phi_1,\phi_2}}|^2,
$$
and since
$$
|S_{\ga_0}\setminus S_{L_{\phi_1,\phi_2}}|\geq  \|\ga_0-L_{\phi_1,\phi_2}\|_*,
$$
we deduce that
$$
|\ga_0|-|L_{\phi_1,\phi_2}|\geq C\|\ga_0-L_{\phi_1,\phi_2}\|_*^2.
$$
By observing that, since $\ga_s\subset S_0$, we have
$$
\|\ga_0-L_{\phi_1,\phi_2}\|_*\geq \|\ga_0-\ga_s\|_*,
$$
combining with \eqref{case1} and \eqref{alpha1} and since we assumed $|\Gamma_s|< |\Gamma_0|$, we find
$$
\alpha \geq C (|\ga_s|-|L_{\phi_1,\phi_2}|)\geq C(|\ga_0|-|L_{\phi_1,\phi_2}|)\geq \|\ga_0-L_{\phi_1,\phi_2}\|_*^2\geq \|\ga_0-\ga_s\|_*^2.
$$

We now consider the case 
\begin{equation}\label{case2}
|\ga_s|-|L_{\phi_1,\phi_2}|<\frac1c(|\ga_0|-|L_{\phi_1,\phi_2}|).
\end{equation}
We will prove \eqref{estGammainotloop} using the second term on the right-hand side of \eqref{alpha1} and \eqref{estnormcurl}.
In this case, $\ga_s$ is ``very close'' to $L_{\phi_1,\phi_2}$, which can be quantified in the following way using the solution of Dido's isoperimetric problem (see \cite{Dido} for a classical survey on the isoperimetric inequality). Since $\ga_s$ and $L_{\phi_1,\phi_2}$ have the same endpoints, applying this inequality we obtain that, for some universal constant $C>0$, we have 
$$
|S_{\ga_s}\setminus S_{L_{\phi_1,\phi_2}}|\leq C( |\ga_s|-|L_{\phi_1,\phi_2}|)^2.
$$
Combining with \eqref{case2}, we deduce that
$$
|S_{\ga_{\phi_1,\phi_2}}\setminus S_{\ga_s}|=\eta|S_0|,
$$ where $\eta=\eta(c)\in (0,1)$ is very close to $1$ provided $c>1$ is large enough.
Therefore, a straightforward computation shows that
$$
\pr{B_0,\ga_{\phi_1,\phi_2}}-\pr{B_0,\ga_s}\geq \frac1{10}(\pr{B_0,\ga_0}-\pr{B_0,\ga_s}),
$$
choosing a fixed but large enough constant $c>1$.
Combining with \eqref{alpha1} and \eqref{estnormcurl}, we find
$$
\alpha\geq C (\pr{B_0,\ga_0}-\pr{B_0,\ga_s})\geq C \|\ga_0-\ga_s\|_*^2.
$$
This concludes the proof of \eqref{estGammainotloop}.

\medskip\noindent
{\bf Step 4. Proof of \eqref{estGammailoop}.} 

\smallskip\noindent
Let $\ga\in X$ be a loop contained in $\overline{B(0,R)^{2\D,+}}$  (defined in \eqref{halfball}) and let $S$ denote the surface enclosed by $\ga$. By Stokes' theorem, we have 
$$
\pr{B_0,\ga}=\int_S\curl B_0\cdot \hat y,
$$
and the isoperimetric inequality therefore implies that if $|\ga|$ is small then $\pr{B_0,\ga}$ is small. For this reason, we assume from now on that $|\ga|>c>0$ for a fixed constant $c>0$.

Now, observe that
\begin{equation}\label{ga1}
\R(\ga_0)-\R(\ga)=\frac{\pr{B_0,\ga_0}-\pr{B_0,\ga}}{|\ga_0|}+\frac{\pr{B_0,\ga}}{|\ga_0||\ga|}\left(|\ga|-|\ga_0|\right).
\end{equation}
By the isoperimetric inequality we deduce that if $\frac1C\leq \left||\ga|-|\ga_0|\right|\leq C$, for some constant $C>1$ close enough to $1$, we have 
$$
\pr{B_0,\ga}\leq \frac{9}{10} \pr{B_0,\ga_0},
$$
or, in words, if the length of the loop is close enough to that of $\ga_0$, then the area enclosed by it is far (enough) from the area of the semicircle (which can of course be made quantitative, but is not really necessary for our purpose). Therefore, by making $C>1$ smaller if necessary (depending on $B_0$ and $R$ only), we deduce from \eqref{ga1} that \eqref{estGammailoop} holds in this case. Moreover, from \eqref{ga1} we immediately see that if $|\ga|-|\ga_0|>C$, then \eqref{estGammailoop} holds. Therefore, from now on we assume that $c<|\ga|<|\ga_0|-c$ for some fixed constant $c>0$ depending on $R$ and $B_0$ only. 

Since $\curl B_0\cdot \hat y\geq 0$ in $B(0,R)^{2\D,+}$ and it is symmetric with respect to the $x$-axis, we can assume without loss of generality that $\ga$ is convex and symmetric with respect to the $x$-axis. Indeed, if $\ga$ is not convex, its convex envelope is a loop which encloses more area and has smaller length than $\ga$, and therefore has a greater ratio than $\ga$. Also, if $\ga$ is not symmetric with respect to the $x$-axis, we consider the loops $\ga_1$ and $\ga_2$, obtained by intersecting $\ga$ with the sets $\{z\geq 0\}$ and $\{z\leq 0\}$ respectively. Without loss of generality, we assume that
$$
\frac{\pr{B_0,\ga_1}}{|\ga_1|-|\ga_1\cap\{z=0\}|}\geq \frac{\pr{B_0,\ga_2}}{|\ga_2|-|\ga_2\cap\{z=0\}|}.
$$
Then, $\tilde \ga$ defined as
$$
\tilde \ga(x,z)=\ga(x,z) \quad \mbox{for } z\geq 0\quad \mbox{and}\quad \tilde \ga(x,z)=\ga(x,-z)\quad \mbox{for } z<0,
$$
is a loop symmetric with respect to the $x$-axis such that $\R(\tilde\ga)\geq \R(\ga)$.

Moreover, since $\curl B_0\cdot \hat y$ is increasing with respect to $x$ in $B(0,R)^{2\D,+}$, we can translate the loop in the $x$-direction until it touches the boundary of $B(0,R)^{2\D,+}$, which once again increases the ratio. 

Since the loop is symmetric with respect to the $x$-axis, we now consider the symmetric sector $S_{\phi,\pi-\phi}$, defined as in Step 3.2 with $\phi_1=\phi\in(0,\pi/2)$ and $\phi_2=\pi-\phi$. Since $|\ga|<|\ga_0|-c$, we deduce that $\phi>\phi_0$, where $\phi_0\in(0,\pi/2)$ is a fixed angle depending on $c$. Moreover, by making $\phi_0$ smaller if necessary, we can assume that $\phi<\pi/2-\phi_0$, because otherwise $\pr{B_0,\ga}\ll 1$ with $|\ga|>c>0$, and therefore \eqref{estGammailoop} holds. 

Since $|\ga|\in (c,|\ga_0|-c)$ and $\phi\in(\phi_0,\pi/2-\phi_0)$, where $c,\phi_0>$ depend on $R$ and $B_0$ only, from the isoperimetric inequality we deduce that there exists $x_0\in (0,R)$ depending on $R$ and $B_0$ only, such that $\ga$ is contained in $B(0,R)^{2\D,+}\cap \{x\geq x_0\}$. We recall that
$$
\frac{\pr{B_0,\ga_{\phi,\pi-\phi}}}{|L_{\phi,\pi-\phi}|}=\sin\left(\frac{\phi+(\pi-\phi)}2\right)\R(\ga_0)=\R(\ga_0),
$$
and
$$
|\ga|\geq |L_{\phi,\pi-\phi}|,\quad \pr{B_0,\ga}\leq \pr{B_0,\ga_{\phi,\pi-\phi}}.
$$
But since $\ga$ is contained in $B(0,R)^{2\D,+}\cap \{x\geq x_0\}$, we deduce that
$$
\pr{B_0,\ga}\leq \pr{B_0,\ga_{\phi,\pi-\phi}}-C,
$$
where $C>0$ is a constant that depends on $R$ and $B_0$ only. 

Hence
$$
\R(\ga)\leq \frac{\pr{B_0,\ga_{\phi,\pi-\phi}}-C}{|L_{\phi,\pi-\phi}|}\leq \R(\ga_0)-C.
$$
This concludes the proof of \eqref{estGammailoop} and the theorem is thus proved.
\end{proof}

We next provide an alternative proof of \eqref{estGammailoop}, that we believe can be extended to the case when $\Omega$ is a flat  ellipsoid, i.e. with larger dimensions in the $xy$ plane than in the $z$-direction, or similar solids of revolution.

\begin{proof}{Alternative proof of \eqref{estGammailoop}} 
Let us assume towards a contradiction that 
$$\sup_{\mathcal C_{\mathrm{loops}}}\R = R_0 = \sup_{\NN}\R.$$
Then for any $\ep>0$ there exists a loop $\ga_\ep$ such that $\R(\ga_\ep)> R_0 - \ep.$  We will obtain a contradiction if $\ep$ is small enough. From now on we omit the subscript $\ep$, understanding that $\ga$ depends on $\ep$.

Define $\phi_1$, $\phi_2$ as above and $\theta =  \phi_2 - \phi_1$. The fact that $\R(\ga)$ tends to $R_0$ as $\ep\to 0$ implies that as $\ep\to 0$, $\liminf\theta>0$. Otherwise we would have $\pr{B_0,\ga}\to 0$, hence $|\ga|\to 0$, and then $\R(\ga)\to 0$ using the isoperimetric inequality as in \cite{Rom2}*{Remark 4.1}.

Let $p$ be one of the furthest points to the origin on $\ga$ (in case there is more than one), and $r = |p|$. Then $\ga$ goes from a point $p_1$ on the ray with angle $\phi_1$ to a point $p_2$ on the ray with angle $\phi_2$ (call this section $\ga_a$) and then goes back from $p_2$ to $p_1$ through $p$. This is depicted in Figure \ref{figure3}. Therefore $|\ga|\ge |\ga_a| + |p_1 - p|+|p_2-p|$, and thus 
$$|\ga|\ge |\ga_a| +\sqrt{(r-r_1)^2+2rr_1(1-\cos\alpha)} + \sqrt{(r-r_2)^2+2rr_2(1-\cos(\theta -\alpha))},$$
where $r_i = |p_i|$ and $\alpha$ is the angle between the vectors $\vec{Op}$ and $\vec{Op_2}$, where $O$ denotes the origin.

\begin{figure}
\centering
\includegraphics[]{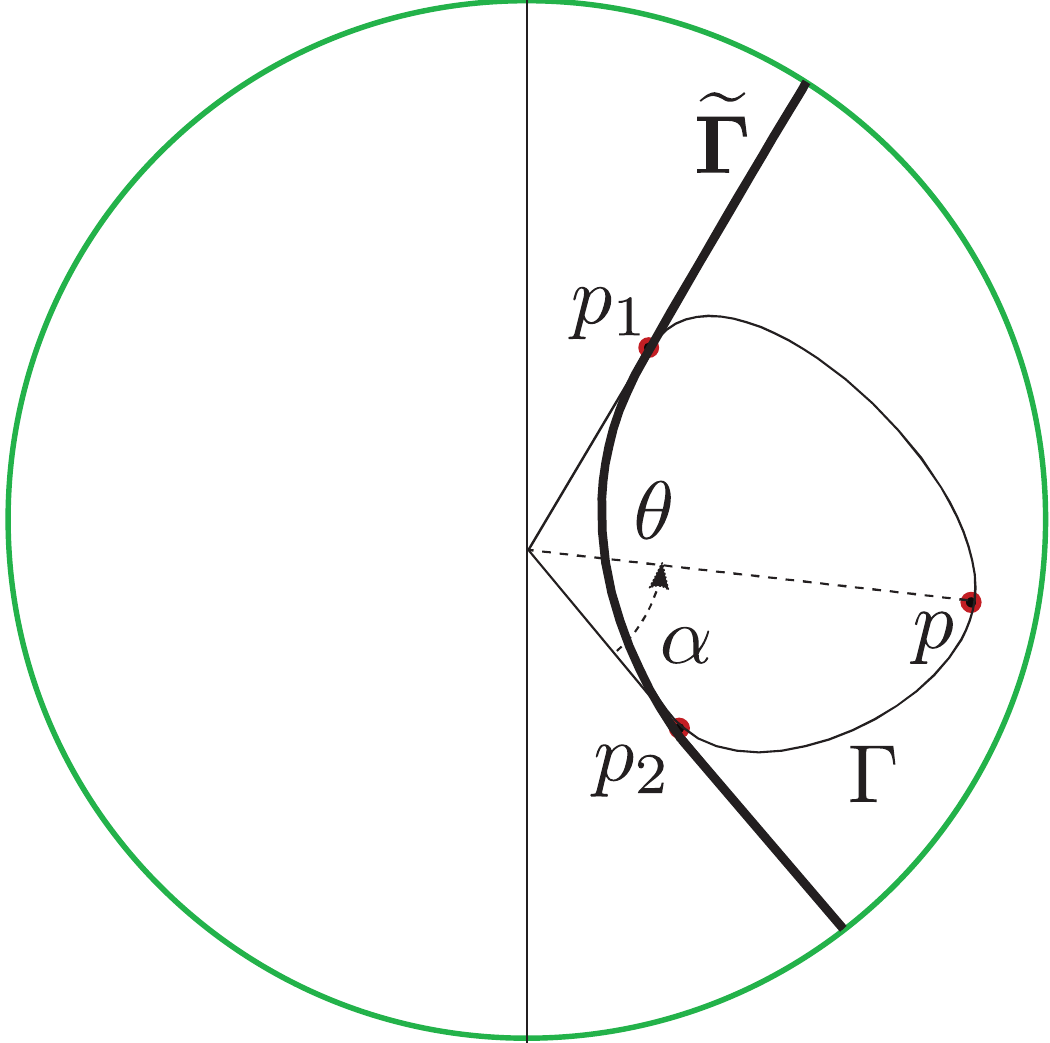}
\caption{Depiction of $\ga$ and $\tilde \ga$.}
\label{figure3}
\end{figure}

We construct a competitor $\tilde\ga$ by replacing the portion $[p_2,p]$ by a straight segment to the boundary on the $\phi_2$-ray, and similarly for the portion $[p,p_1]$. Then
$$|\tilde\ga| = |\ga_a|+ (R-r_1) + (R-r_2).$$
Moreover, $\pr{B_0,\tilde\ga}\ge\pr{B_0,\ga}$, since $\curl B_0\cdot \hat y\ge 0$, so that 
$$R_0 \ge \R(\tilde\ga)\ge\frac{|\ga|}{|\tilde\ga|}R(\ga)\ge \frac{|\ga|}{|\tilde\ga|} (R_0 - \ep).$$
It follows that 
$$|\tilde \ga|\ge |\ga| \(1 - \frac\ep{R_0}\).$$ 

As $\ep\to 0$, since $\liminf\theta>0$, we may assume  by letting $\ep\to 0$ along a well chosen sequence that either $\alpha$ or $\theta - \alpha$ remains bounded away from $0$. Assume w.l.o.g that it is $\alpha$. Then, in view of the lower bound  for $|\ga|$ and the expression of  $|\tilde\ga|$, we deduce that, as $\ep\to 0$ along this sequence, either $\liminf r < R$ or $\liminf r_1 = 0$. If we assumed $\theta - \alpha$ was bounded away from $0$, then $r_2$ would replace $r_1$ in the alternative.

In the first case, we obtain a contradiction by shifting $\ga$ to the right by a positive amount independent of $\ep$, which increases the ratio by a positive amount as well, since $\curl B_0\cdot \hat y$ is increasing w.r.t. the $x$ variable. This contradicts the fact that $\R(\ga)$ tends to $R_0$ as $\ep\to 0$.

Therefore we may assume, going to a further subsequence, that $r_1\to 0$ as $\ep\to 0$ and $r\to R$. This implies that $\liminf |\ga|\ge 2R=|\ga_0|$.Taking the diameter as a competitor, this implies that   $\liminf\pr{B_0,\ga}\ge \pr{B_0,\ga_0}$, which implies that the area enclosed by $\ga$ tends to that of the half-disk. Since $\ga$ is a closed curve, this implies using  the isoperimetric inequality that $\liminf|\ga|\ge \sqrt 2\pi R$. This contradicts the near optimality of $\ga$ as $\ep\to 0$ since $2\pi R > |\ga_0|$.
\end{proof}

We end this section by providing a proof for Proposition \ref{prop:positivity}.
\begin{proof}
We begin by defining 
$$
\dot H^1(\RR^3,\RR^3)=\{A\in L^6(\RR^3,\RR^3)\ | \ \nabla A \in L^2(\RR^3,\RR^3)\}
$$
and the subspace
$$
\dot H^1_{\diver=0}\colonequals \{A\in \dot H^1(\RR^3,\RR^3) \ | \ \diver A=0\ \mathrm{in}\ \RR^3\},
$$
in which one has 
\begin{equation*}%\label{space2}
\|A\|_{\dot H^1_{\diver=0}}\colonequals \|\nabla A\|_{L^2(\RR^3,\RR^3)}= \|\curl A\|_{L^2(\RR^3,\RR^3)}.
\end{equation*}
Letting $A_{0,\ex}=(-y/2,x/2,0)$, we recall from \cite{Rom2}*{Section 3} that $A_0$ is the unique minimizer of the functional 
$$
J(A)=\int_\Omega |\curl B_A|^2+\int_{\RR^3}|\curl A-H_{0,\ex}|^2
$$
in the space $\left(A_{0,\ex}+\dot H^1_{\diver=0},\|\, \cdot \, \|_{\dot H^1_{\diver=0}}\right)$, which satisfies the Euler--Lagrange equation
\begin{equation}\label{relBH}
\curl H_0+\curl B_0\indic_\Omega=0\quad \mathrm{in}\ \RR^3.
\end{equation}
Here we used the fact that $\curl H_{0,\ex}=0$ and the notation $\curl A_0\equalscolon H_0$. In particular, by elliptic regularity, one easily deduces that $A_0\to A_{0,\ex}$ uniformly at infinity.

Since $A_0=\curl B_0+\nabla \phi_0$ in $\Omega$, we can rewrite the previous equation as 
$$
-\Delta A_0=(\nabla \phi_0-A_0)\indic_\Omega \quad \mathrm{in}\ \RR^3.
$$

%We begin by recalling the computation of $B_0$, which starts with the expression of the Ginzburg--Landau energy when $u = e^{i\phi}$ has modulus one. 
%$$GL(e^{i\phi},A) = \hal\int_\Omega |\nab\phi - A|^2+\hal\int_{\RR^3}|\curl A - H_\ex|^2.$$
%Using the vector identity $\diver(X\times Y) = \curl X\cdot Y - X\cdot\curl Y$, we find that the Euler--Lagrange equation obtained from variations of $A$ satisfied by the minimizer $(e^{i\phi_0},A_0)$ is 
%$$\curl (H_0-H_\ex) = (\nab\phi_0 - A_0)\indic_\Omega,\quad H_0 = \curl A_0.$$

%We assume $H_\ex = h_\ex(0,0,1)$. Then $A_0$ belongs to $A_\ex+\dot H^1$, where $A_\ex = \frac{h_\ex}2(-y,x,0)$ and $\dot H^1$ is the completion of smooth compactly supported vector fields for the $L^2$ norm of the derivatives. In this sense, $A_0\to A_\ex$ at infinity.

%We will further assume that we are in the Coulomb gauge, so that $\diver A_0 = 0$ in $\RR^3$. Taking the divergence  of the equation $\curl H_0 = (\nab\phi_0 - A_0)\indic_\Omega$, we then find that $\Delta \phi_0 = 0$, and replacing $H_0$ by $\curl A_0$, that 
%$$ - \Delta A_0 =  (\nab\phi_0 - A_0) \indic_\Omega.$$

From now on we assume that $\Omega$ is axisymmetric with respect to $\hat z$. Let $(r,\theta,z)$ be cylindrical coordinates in $\RR^3$ and $\Omega$ be the set of points whose $z$ coordinate belongs to some open interval, and whose $r$ coordinate satisfies $r<r_\Omega(z)$. 

The uniqueness of $A_0$ implies that it is axisymmetric as well. In particular $\phi_0$ is independent of $\theta$. If we write $A_0$ as a one form $A_0 = A_0^\theta\,d\theta+A_0^r\,dr+A_0^z\,dz$, then the components are functions independent of $\theta$. We may then compute the exterior differential, whose $ dr\wedge d\theta$ component is $\partial_r A_0^\theta$, so that 
$$\curl A_0\cdot\hat z = \frac1 r\partial_r A_0^\theta,\quad \partial_\theta A_0^\theta = \partial_\theta \phi_0 = 0.$$

Then, we compute $\Delta(A_0\cdot\hat\theta)=\hat\theta\cdot \Delta A_0 + \curl A_0\cdot\hat z$ to find that 
$$ - \Delta A_0^\theta + A_0^\theta\indic_\Omega = \frac1r\partial_r A_0^\theta.$$
Assuming the minimum of $A_0^\theta$ is achieved at $p$, we find that $A_0^\theta(p)$ cannot be negative. Since $A_0^\theta$ is positive at infinity, we deduce that $A_0^\theta$ is nonnegative in $\RR^3$.

Since $\curl H_0 = (\nab\phi_0 - A_0)\indic_\Omega$, and recalling that  $\partial_\theta \phi_0 = 0$, we deduce that 
$$ \curl H_0\cdot\hat\theta\le 0\quad \mathrm{in} \ \RR^3.$$
Finally, using \eqref{relBH}, we obtain
$$ \curl B_0\cdot\hat\theta\ge 0\quad \mathrm{in}\ \Omega.$$

%The field $B_0$ is defined in $\Omega$ only, as follows. Assume that $(A_0, \phi_0)$ are as above, and moreover that $\diver A_0 = 0$. 
%
%The restriction of $A_0$ may be decomposed as $A_0 = \curl B_0 + \nabla\varphi_0$, where $\diver B_0 = 0$, and where the boundary conditions $B_0\times \nu = 0$ and $\nu\cdot \nabla\varphi = \nu\cdot A_0$ are satisfied on the boundary. The decomposition is unique. Taking the divergence of the decomposition, we find $\Delta\varphi_0 = 0$.
%
%From the equation $\curl H_0 = (\nab\phi_0 - A_0)\indic_\Omega$, taking its divergence, we obtain $\Delta\phi_0 = 0$ in $\Omega$, and $\nu\cdot\nabla\phi_0 = \nu\cdot A_0$ on the boundary. It follows that $\varphi_0 = \phi_0$. Therefore we have $\curl H_0 = -\curl B_0$ in $\Omega$ and therefore 

%Note also that $B_0$ has the property that for any $A$ going to zero at infinity we have 
%$$\int_\Omega B_0\cdot\curl A = \int_{\partial\Omega} (A\times B_0)\cdot\nu +\int_\Omega\curl B_0\cdot A = - \int_\Omega (\nab\phi_0 - A_0)\cdot A.$$
%Since 
%$$\int_{\RR^3} (H_0-H_\ex)\cdot\curl A =\int_{\RR^3} \curl(H_0-H_\ex)\cdot A =\int_\Omega (\nab\phi_0 - A_0)\cdot A,$$
%it follows that that for any $A$ going to zero at infinity we have 
%$$\int_{\RR^3} (H_0-H_\ex)\cdot\curl A = -\int_\Omega B_0\cdot\curl A.$$
\end{proof}

\section{Bounded vorticity}\label{sec4}
In the rest of the paper $C>0$ denotes a positive constant which may change from line to line, independent of $\ep$ but which may depend on $\Omega$ and $B_0$, hence on $\Gamma_0$ as well. 
\begin{lemma}\label{lem:controllength} 
Given $\ga\in X$, by letting 
$$
\al=\R(\ga_0)-\R(\ga) \ge 0 ,
$$
we have
$$
\left||\ga|-|\ga_0|\right|\leq \frac1{\R(\ga_0)}(\alpha|\ga|+  \|B_0\|_{C_T^{0,1}(\Omega,\RR^3)} \|\ga-\ga_0\|_*).
$$
\end{lemma}
\begin{proof}
By noting that
$$
\al=\pr{B_0,\ga_0}\left(\frac1{|\ga_0|}-\frac1{|\ga|} \right)+\frac{\pr{B_0,\ga_0}-\pr{B_0,\ga}}{|\ga|}
$$
and 
$$
\left|\pr{B_0,\ga_0}-\pr{B_0,\ga}\right|\leq \|B_0\|_{C_T^{0,1}(\Omega,\RR^3)} \|\ga-\ga_0\|_*,
$$
we get
$$
\left| |\ga|-|\ga_0| \right|\leq \frac{|\ga_0|}{\pr{B_0,\ga_0}}\left(\al|\ga|+ \|B_0\|_{C_T^{0,1}(\Omega,\RR^3)} \|\ga-\ga_0\|_*\right).	
$$
\end{proof}
\begin{lemma}\label{lem:tubular} 
Let $\ga\in X$ not be a loop. There exists $\de_0>0$, depending on $\Gamma_0$  such that, for any $0<\de<\de_0$, if
$$
 \|\ga-\ga_0\|_*\leq \de \quad  \mbox{and} \quad  ||\ga|-|\ga_0||\leq \de,
$$
then $\ga$ belongs to the tubular neighborhood of $\ga_0$ of thickness $C\sqrt\de$, where $C>0$  is universal.
\end{lemma}
\begin{proof}
For any $\eta>0$, let $\mathcal T_\eta$ denote the tubular neigborhood of $\Gamma_0$ of thickness $\eta$. We claim that 
\begin{equation}\label{Tgeta}
|\Gamma \cap \mathcal T_{\sqrt \delta}  |\ge (1- \frac{2}{|\Gamma_0|} \sqrt \delta) |\Gamma|.\end{equation} 
Indeed, if not,  we would be able to construct a vector field  $B\in C_0^{0,1}(\Omega,\RR^3)$ supported   in $\mathcal T_{\sqrt\delta} $  with $\|B\|_\infty\le  \sqrt \de$ and $\|B\|_{C_0^{0,1}(\Omega,\RR^3)}\leq 1$, such that $\pr{B,\ga_0}=   \sqrt \delta |\Gamma_0| $. This way we would have $$|\pr{B,\ga  }| \le \sqrt \delta | \Gamma \cap \mathcal T_{\sqrt \delta}  |\le \sqrt \delta  ( 1- \frac2{|\Gamma_0|} \sqrt \delta)  |\Gamma| \le \sqrt \delta |\Gamma_0|- \frac2{|\Gamma_0|}  \delta |\Gamma_0|+ O(\delta^{\frac32}) $$
and thus if $\delta $ is small enough, by definition of $\|\cdot \|_*$ we would deduce
$$\|\Gamma_0- \Gamma\|_* > \delta ,$$ a contradiction. Thus, we deduce from \eqref{Tgeta} that
$$|\Gamma \cap (\mathcal T_{\sqrt \delta} )^c |\le \frac2{|\Gamma_0|} \sqrt \delta |\Gamma|\le  \frac{2}{|\Gamma_0|} \sqrt \delta ( |\Gamma_0|+ \delta) $$ which implies, if $\delta_0<|\Gamma_0|$ that 
$\mathrm{dist} \, ( \Gamma, \mathcal T_{\sqrt \delta}) \le 4 \sqrt \delta $, which yields the result.
\end{proof}

\subsection{\texorpdfstring{$\ep$}{Epsilon}-level estimates}
We now recall the $\ep$-level estimates provided in \cite{Rom}.
\begin{theoremm}\label{theorem:epslevel} For any $m,n,M>0$ there exist $C,\ep_0>0$ depending only on $m,n,M,$ and $\partial\Omega$, such that, for any $\ep<\ep_0$, if $(u_\ep,A_\ep)\in H^1(\Omega,\C)\times H^1(\Omega,\RR^3)$ is a configuration such that $F_\ep(u_\ep,A_\ep)\leq M|\log\ep|^m$ then there exists a polyhedral $1$-dimensional current $\nu_\ep$ such that
\begin{enumerate}[leftmargin=*,font=\normalfont]
\item $\nu_\ep /2\pi$ is integer multiplicity.
\item $\partial \nu_\ep=0$ relative to $\Omega$,
\item $\mathrm{supp}(\nu_\ep)\subset S_{\nu_\ep}\subset \overline \Omega$ with $|S_{\nu_\ep}|\leq C|\log\ep|^{-q}$, where $q(m,n)\colonequals\frac32 (m+n)$,
\item 
\begin{multline*}
\int_{S_{\nu_\ep}}|\nabla_{A_\ep} u_\ep|^2+\frac{1}{2\ep^2}(1-|u_\ep|^2)^2
\geq  |\nu_\ep|(\Omega)\left(\log \frac1\ep-C \log \log \frac1\ep\right)-\frac{C}{|\log\ep|^n},
\end{multline*}
\item and for any $\gamma\in(0,1]$ there exists a constant $C_\gamma$ depending only on $\gamma$ and $\partial\Omega$, such that 
\begin{equation}\label{EstimateJ0}
\|\mu(u_\ep,A_\ep) -\nu_\ep\|_{C_T^{0,\gamma}(\Omega)^*}\leq C_\gamma \frac{F_\ep(u_\ep,A_\ep)+1}{|\log \ep|^{q\gamma}}.
\end{equation}
\end{enumerate}
\end{theoremm}
We point out that the lower bound part of the previous theorem is slightly different than the one in \cite{Rom}*{Theorem 1.1}. Nevertheless, this inequality is proved within the proof of the latter theorem in \cite{Rom}*{Section 8}. 

It is also convenient to state the following result, whose proof we postpone to Appendix \ref{appendix}, which shows that the vorticity estimate holds in the flat norm and that will be needed in the following section. 
\begin{lemma}\label{lemma:vorticityflat} Under the hypothesis of Theorem \ref{theorem:epslevel}, it holds that, for any $\alpha\in (0,1)$,
\begin{equation*}%\label{EstimateJ} 
\|\mu(u,A) -\nu_\ep\|_{\F(\Omega_\ep)}\leq C \frac{F_\ep(u,A)+1}{|\log \ep|^{\alpha q}},
\end{equation*}
where 
$$
\Omega_\ep\colonequals \left\{x\in \Omega \ | \ \mathrm{dist}(x,\partial\Omega)> |\log\ep|^{-(q+1)}\right\}.
$$
\end{lemma}

\subsection{Energy-splitting}
The proof of Theorem \ref{teo:boundedvorticity} relies on the following expansion of the energy, which is directly obtained by combining \eqref{Energy-Splitting} with the $\ep$-level estimates in Theorem \ref{theorem:epslevel}.
\begin{proposition}\label{prop:energysplitting}
Let $(\u,\A)\in H^1(\Omega,\C)\times [A_\ex+H_{\curl}]$ be a configuration such that $GL_\ep(\u,\A)\leq M|\log\ep|^m$ for some $m,M>0$, and define $u=e^{-ih_\ex\phi_0}\u$ and $A=\A-h_\ex A_0$, where $(e^{ih_\ex\phi_0},h_\ex A_0)$ is the approximation of the Meissner solution. Then, letting $\nu_\ep=\sum_{i\in I_\ep}2\pi \ga_i^\ep$ be the vorticity approximation of the configuration $(u,A)$ given by Theorem \ref{theorem:epslevel} (with $n$ sufficiently large), we have
\begin{multline}\label{energysplitting}
GL_\ep(\u,\A)\geq h_\ex^2 J_0+ \pi \sum_{i\in I_\ep}|\ga_i^\ep|\left(\log \frac1\ep -C\log \log \frac1\ep \right)\\
-2\pi h_\ex\sum_{i\in I_\ep}\pr{B_0,\ga_i^\ep} +\E(u,A)+o(1),
\end{multline}
where $\E$ denotes the excess energy, defined by
$$
\E(u,A)=\frac12 \int_{\Omega\setminus S_{\nu_\ep}}|\nabla_{A}u|^2+\frac1{2\ep^2}(1-|u|^2)^2+\frac12\int_{\RR^3}|\curl A|^2.
$$
\end{proposition}

\subsection{Proof of Theorem \ref{teo:boundedvorticity}}
\begin{proof} We proceed in several steps.

\medskip\noindent
{\bf Step 1 : Energy comparison and deducing information on the curves. }
First we claim that 
\begin{equation}
\label{majoF}
F_\ep(u, A) \le C \lep^2\end{equation} 
for some $C>0$ independent of $\ep$.
Indeed, by assumption, we have
\begin{equation}\label{boundJ0}
GL_\ep(\u,\A)\leq GL_\ep (e^{ih_\ex\phi_0},h_\ex A_0)=h_\ex^2 J_0.
\end{equation}
On the other hand, by gauge invariance, we have
\begin{align*}
F_\ep(u,A)=F_\ep(\u,\A-h_\ex\curl B_0) &\leq 2F_\ep(\u,\A)+2F_\ep(1,h_\ex \curl B_0)\\ &\leq 2F_\ep(\u,\A)+Ch_\ex^2.
\end{align*}
We may then insert this and $h_\ex \le C \lep$  into \eqref{Energy-Splitting} to obtain \eqref{majoF} using the vorticity estimate \eqref{EstimateJ0}.

Next, combining  \eqref{boundJ0} with \eqref{energysplitting} 
we obtain that
\begin{equation}\label{in1}
0\geq \pi \sum_{i\in I_\ep}|\ga_i^\ep|\left(\log \frac1\ep -C\log \log \frac1\ep \right)-2\pi h_\ex \sum_{i\in I_\ep}\pr{B_0,\ga_i^\ep} +\E(u,A)+o(1).
\end{equation}
Let us define, for any $i\in I_\ep$,
$$
\alpha_i^\ep\colonequals \R(\ga_0)-\R(\ga_i^\ep),
$$
and note that $|\alpha_i^\ep|\leq 2\R(\ga_0)$. 

Since we assume that $h_\ex=H_{c_1}^0+O(\log|\log \ep|)=\dfrac1{2\R(\ga_0)}|\log \ep|+O(\log|\log\ep|)$, by definition of $\R$  we have
\begin{multline*}
-h_\ex\sum_{i\in I}\pr{B_0,\ga_i^\ep}=h_\ex \sum_{i\in I_\ep}|\ga_i^\ep| \left(-\R(\ga_0)+ \alpha_i^\ep\right)\\
=-\frac{|\log \ep|}2 \sum_{i\in I_\ep}|\ga_i^\ep|+ H_{c_1}^0\sum_{i\in I_\ep}|\ga_i^\ep|\alpha_i^\ep+O\left(\log|\log \ep|\sum_{i\in I_\ep}|\ga_i^\ep|\right).
\end{multline*}
Plugging this into \eqref{in1}, we are led to
\begin{equation}\label{in2}
O\left(\log|\log \ep|\sum_{i\in I_\ep}|\ga_i^\ep|\right)\geq 2\pi H_{c_1}^0\sum_{i\in I_\ep}|\ga_i^\ep|\alpha_i^\ep+\E(u,A)+o(1).
\end{equation}
Let $0<\de\ll 1$ and define 
$$
I_\good^{\ep,\de} \colonequals \{i\in I_\ep \ | \ \alpha_i^\ep<\de^N \},\quad I_\bad^{\ep,\de} \colonequals \{i\in I_\ep \ | \ \alpha_i^\ep\geq\de^N \}.
$$
By combining Lemma \ref{lem:controllength} and Condition \ref{nondegencond}, we deduce that if $\delta$ is small enough (depending on $\|B_0\|_{\infty}$ and $|\Gamma_0|$),
$$
\frac12|\ga_0|\leq|\ga_i^\ep|\leq 2|\ga_0|\quad \mathrm{for\ any}\ i\in I_\good^{\ep,\de}. 
$$
In particular, we have
$$
\sum_{i\in I_\ep}|\ga_i^\ep|\alpha_i^\ep\geq \frac{|\ga_0|}2 \sum_{i\in I_\good^{\ep,\de}} \alpha_i^\ep +\sum_{i\in I_\bad^{\ep,\de}}|\ga_i^\ep|\de^N.
$$
In addition, by letting $N_\good^{\ep,\de}\colonequals \#(I_\good^{\ep,\de})$, we have
$$
\sum_{i\in I_\ep} |\ga_i^\ep|\leq 2|\ga_0|N_\good^{\ep,\de} +\sum_{i\in I_\bad^{\ep,\de}}|\ga_i^\ep|.
$$
By inserting the last two estimates into \eqref{in2}, we deduce that
\begin{multline}\label{ineqwithde}
\de^N \sum_{i\in I_\bad^{\ep,\de}} |\ga_i^\ep|+\frac{|\ga_0|}2 \sum_{i\in I_\good^{\ep,\de}}\alpha_i^\ep+\frac{\E(u,A)}{|\log\ep|}+o(1)\\
\leq C\frac{\log|\log \ep|}{|\log \ep|}\left(N_\good^{\ep,\de}+\sum_{i\in I_\bad^{\ep,\de}} |\ga_i^\ep|\right).
\end{multline}
Hence, choosing
\begin{equation}\label{choicede1}
\de=\de(\ep)=\left[(C+\eta)\frac{\log|\log\ep|}{|\log\ep|}\right]^\frac1N
\end{equation}
for some $\eta>0$ to be fixed later, we deduce that
\begin{align}
\sum_{i\in I_\good^{\ep,\de}}\alpha_i^\ep&\leq C\frac{\log|\log \ep|}{|\log \ep|} N_\good^{\ep,\de},\notag\\
\sum_{i\in I_\bad^{\ep,\de}}|\ga_i^\ep|&\leq C\eta^{-1}N_\good^{\ep,\de},\label{bad}\\
\E(u,A)&\leq C\log|\log\ep|N_\good^{\ep,\de}\label{excess}.
\end{align}

\medskip\noindent
{\bf Step 2:  Lower bounds on annuli. }
The $\ep$-level lower bound derived from the three dimensional vortex approximation construction introduced in \cite{Rom}, captures well the energy which lies very near the vortices, but misses the repulsion energy between vortices which accumulate locally around the curve $\Gamma_0$. The missing energy can be recovered from the excess energy $\E(u,A)$ by combining a suitable slicing procedure with the method of ``lower bounds on annuli'', introduced in \cite{SanSer2} and re-used in \cites{SanSerBook,SanSer4} in 2D.
We next describe the slicing procedure. 

\medskip
Let us assume $\ga_0:[0,\lga]\to \RR^3$ is a $C^2$ injective open curve parametrized by arclength and define its regular tubular neighborhood $U$ of thickness $c_0>0$ by
$$
U=\{x+y \ | \ x\in \ga_0,\ y\perp \t_0(x),\ |y|<c_0\},
$$
where $\t_0(x)$ denotes the unit tangent vector to $\ga_0$ at $x$ and $c_0$ is sufficiently small so that the normal projection $p:U\to \ga_0([0,\lga])$ is well defined and $C^1$.

We now define the $C^1$ map $\psi\colonequals (\ga_0)^{-1}\circ p$. %and denote $\mathrm{Lip}(\psi)$ its Lipschitz constant. 
By the coarea formula, letting $U'\colonequals (U\cap\Omega)$, we have
\begin{align}
\E(u,A)&\geq \frac12 \int_{U'\setminus S_{\nu_\ep}}\left(|\nabla_{A}u|^2+\frac1{2\ep^2}(1-|u|^2)^2\right)\frac{|\nabla \psi|}{\|\nab \psi\|_{L^\infty}}+\frac12\int_{U'}|\curl A|^2\frac{|\nabla \psi|}{\|\nab \psi\|_{L^\infty} }\label{circles1}\\
&\geq \frac1{2\|\nab \psi\|_{L^\infty}} \Big(\int_0^{|\Gamma_0|}\int_{(\psi^{-1}(s)\cap U')\setminus S_{\nu_\ep}}|\nabla_{A}u|^2+\frac1{2\ep^2}(1-|u|^2)^2d\H^2ds\notag\\
&\hspace*{7cm}+\int_0^{|\Gamma_0|} \int_{\psi^{-1}(s)\cap U'}|\curl A|^2d\H^2ds\Big).\notag
\end{align}
Let us observe that, by definition of $p$, the level sets of $\psi^{-1}$ are flat. 

\medskip
Hereafter, for any $s\in(0,\lga)$ and $t>0$, we denote by $B_{s,t}$ the two-dimensional ball with center $(\ga_0)(s)$ and radius $t$ contained in $\psi^{-1}(s)$. We now use the method of ``lower bounds on annuli'', whose main ingredient is the following estimate (see  \cite{SanSer4}*{(3.18)}). Given a ball $B_{s,t}$, if $|u|\geq \frac12$ on $\partial B_{s,t}$ then, for some constant $c>0$,
\begin{equation}\label{circle}
\frac12\int_{\partial B_{s,t}}|\nabla_{A}u|^2+\frac1{2\ep^2}(1-|u|^2)^2+\frac12\int_{B_{s,t}}|\curl A|^2\geq c\frac{d_{s,t}^2}{t},
\end{equation}
where $d_{s,t}$ is the degree of $u/|u|$ on $\partial B_{s,t}$. To apply this estimate, the circles $\partial B_{s,t}$'s must avoid the set $\{|u|\leq 1/2\}$ for most $s$'s and $t$'s and must be completely included in $\Omega$. To guarantee this, we now remove the ``bad" slices.  

Let us observe that by \eqref{majoF}, the coarea formula and the Cauchy Schwarz inequality, we have
\begin{align*}
C \lep^2 \ge F_\ep(u,A)\geq C\int_\Omega|\nabla|u||^2+\frac1{\ep^2}(1-|u|^2)^2&\geq C\int_\Omega \frac{|\nabla|u||(1-|u|^2)}\ep\\
&=C\int_0^\infty\int_{\{|u|=\sigma\}}\frac{(1-\sigma^2)}\ep d\H^2d\sigma.
\end{align*}
Therefore, there exists $\sigma_0\in[1/2,5/8]$ such that
$$
\H^2(\{|u|=\sigma_0\})\leq \ep^\frac34.
$$
By the coarea formula, we deduce that
$$
\S_1\colonequals \{s\in (0,\lga) \ | \ \H^1\left(\{|u|=\sigma_0\}\cap \psi^{-1}(s)\cap U' \right)\leq \ep^\frac12\}
$$
satisfies $|\S_1|\geq \lga -\ep^\frac14$, that is for most $s$'s and $t$'s the circles $\partial B_{s,t}$ can avoid the set $\{|u|\leq 1/2\}$.

\medskip
We also need a control on the area of $\partial S_{\nu_\ep}$, where $S_{\nu_\ep}$ is the set provided by Theorem \ref{theorem:epslevel}. From \cite{Rom}*{Lemma 5.5}, it is easy to see that 
$$
\H^2(\partial S_{\nu_\ep})\leq |\log\ep|^{-2}
$$
if $n$ is chosen sufficiently large (in terms of $m$). By using the coarea formula once again, we deduce that 
$$
\S_2\colonequals \{s\in (0,\lga) \ | \ \H^1\left(\partial S_{\nu_\ep}\cap \psi^{-1}(s)\cap U'\right)\leq |\log \ep|^{-1} \}
$$
is such that $|\S_2|\geq \lga-|\log\ep|^{-1}$.

Let us define
$$
\S_3\colonequals \{s\in(0,\lga) \ | \ B_{s,|\log\de|^{-2}}\subset \Omega \}
$$
and observe that there exists a constant $C=C(\partial\Omega)$ such that $|\S_3|\geq \lga -C(\partial\Omega)|\log\de|^{-2}$.

Finally, let us consider the set
$$
\S_4\colon= \left\{s\in(0,\lga) \ \vline \ \int_{\psi^{-1}(s)\cap U'}|\nabla_{A}u|^2+\frac1{2\ep^2}(1-|u|^2)^2+|\curl A|^2\leq C|\log \ep|^3 \right\}
$$
and observe that by the coarea formula, reasoning as for \eqref{circles1}, since \eqref{majoF} holds, we have  $|\S_4|\geq \lga -C|\log \ep|^{-1}$.

From the previous estimates we deduce, in particular, that $\S\colonequals  \cup_{i=1}^4 \S_i$ satisfies 
\begin{equation}\label{measureZ}
|\S|\geq \lga-C|\log\de|^{-2}.
\end{equation}
Let $s\in \S$, define 
$$
\T_s\colonequals \{t\in (\de^\frac12,|\log\de|^{-2}) \ | \ \partial B_{s,t}\cap (\{|u|\leq 1/2\}\cup S_{\nu_\ep})=\varnothing\},
$$
and observe that, for any $t\in \T_s$, we have \eqref{circle}. 

\medskip
Let us now observe that, since $|\T_s|\leq 1$, we have
\begin{align}
\int_0^{\lga}&\int_{(\psi^{-1}(s)\cap U')\setminus S_{\nu_\ep}}|\nabla_{A}u|^2+\frac1{2\ep^2}(1-|u|^2)^2d\H^2ds+\int_0^{\lga}\int_{\psi^{-1}(s)\cap U'}|\curl A|^2d\H^2ds\label{circles2}\\
&\geq \int_\S \int_{\T_s}\int_{\partial B_{s,t}}|\nabla_{A}u|^2+\frac1{2\ep^2}(1-|u|^2)^2d\H^1dtds+\int_\S \int_{\T_s}\int_{B_{s,t}}|\curl A|^2d\H^1dtds\notag\\
&\geq 2c\int_\S \int_{\T_s} \frac{d_{s,t}^2}t dtds.\notag
\end{align}

By Lemma \ref{lem:controllength}, Condition \ref{nondegencond}, and definition of $I_\good^{\ep,\de}$, we have that, for any $i\in I_\good^{\ep,\de}$,
$$
\|\ga_i^\ep-\ga_0\|_*\leq C\de \quad \mathrm{and}\quad \left||\ga_i^\ep|-|\ga_0| \right|\leq C\de.
$$
This implies from Lemma \ref{lem:tubular} that all ``good'' curves are contained in a tubular neighborhood of $\ga_0$ of thickness $C \de^\frac12$.

Since $s\in \S\subseteq \S_4$, evaluating the degree in the slice in two different ways, we have
$$
d_{s,t}=N_\good^{\ep,\de}+d_\bad^{s,t},
$$
where $d_\bad^{s,t}$ is the degree of the ``bad'' curves on $\partial B_{s,t}$.  We immediately see that
$$
|d_{s,t}|\geq N_\good^{\ep,\de}-n_\bad^s
$$ 
with $n_\bad^s$ being the number of bad curves on the slice $\psi^{-1}(s)\cap U'$.
 Moreover, since $s\in \S\subseteq \S_1$,
$$
\int_\S d_{s,t}^2ds\geq \int_\S (N_\good^{\ep,\de}-n_\bad^s)^2ds \geq |\S|(N_\good^{\ep,\de})^2-2N_\good^{\ep,\de}\sum_{i\in I_\bad^{\ep,\de}}|\ga_i^\ep|.
$$
This combined with \eqref{bad} and \eqref{measureZ} gives
\begin{equation}\label{estimateZ}
\int_\S (N_\good^{\ep,\de}-n_\bad^s)^2ds \geq (|\S|-2\eta^{-1})(N_\good^{\ep,\de})^2\geq\frac{|\ga_0|}2(N_\good^{\ep,\de})^2,
\end{equation}
where we have chosen (and fixed) $\eta=5|\ga_0|^{-1}$.

On the other hand, since $s\in \S \subset \S_1\cup \S_2$ and the function $1/t$ is decreasing, we have
$$
\int_{\T_s}\frac1t dt\geq \int_{\de^\frac12+\ep^\frac12+|\log\ep|^{-1}}^{|\log\de|^{-2}}\frac1t dt\geq \frac12 |\log\de^\frac12|,
$$if $\ep$, hence $\delta$ is small enough.
By combining this with \eqref{estimateZ}, we obtain
$$
\int_\S \int_{\T_s} \frac{d_{s,t}^2}t dtds\geq \frac{|\ga_0|}4 (N_\good^{\ep,\de})^2 |\log \de^\frac12|.
$$
Combining with \eqref{circles1} and \eqref{circles2}, this yields
\begin{equation}\label{excesscircles}
\E(u,A)\geq C (N_\good^{\ep,\de})^2 |\log \de|.
\end{equation} 

\medskip\noindent
{\bf Step 3:  Bound on $N_\good^{\ep,\de}$. }  By combining \eqref{excess} and \eqref{excesscircles}, we are led to
$$
N_\good^{\ep,\de} |\log \de|\leq C \log |\log \ep|.
$$
But, by our choice of $\de$ (recall \eqref{choicede1}), we have
$$
\frac{\log |\log \ep|}{|\log \de|}\leq C,
$$
which in turn allows us to conclude that $N_\good^{\ep,\de} \leq C$, that is the number of ``good'' curves is bounded independently of $\ep$. Inserting this estimate into \eqref{bad}, we also deduce that
$$
\sum_{i\in I_\bad^{\ep,\de}}|\ga_i^\ep|\leq C.
$$

\medskip\noindent{\bf Step 4: Conclusion.} Let $\alpha\in(0,1)$. We now choose 
$$
\de_1=\left(\frac{\log|\log\ep|}{|\log\ep|^\alpha}\right)^{\frac1{N}}.
$$
Since $\de_1>\de$ (for $\ep$ sufficiently small), we know that
$$
I_\bad^{\ep,\de_1}\subset I_\bad^{\ep,\de},
$$
which implies that 
$$
\sum_{i\in I_\bad^{\ep,\de_1}}|\ga_i^\ep|\leq \sum_{i\in I_\bad^{\ep,\de}}|\ga_i^\ep|\leq C.
$$ 
Let us now observe that the right-hand side of \eqref{ineqwithde} with $\de$ replaced by $\de_1$ is bounded by 
$$
C\frac{\log |\log \ep|}{|\log \ep|}(N_\good^{\ep,\de_1}+C).
$$
From this, we deduce that
\begin{align}
\sum_{i\in I_\good^{\ep,\de_1}}\alpha_i^\ep&\leq C\frac{\log |\log \ep|}{|\log \ep|}(N_\good^{\ep,\de_1}+C)\notag,\\
\sum_{i\in I_\bad^{\ep,\de_1}}|\ga_i^\ep|&\leq \de_1^{-N}C\frac{\log |\log \ep|}{|\log \ep|}(N_\good^{\ep,\de_1}+C)\label{bad1},\\
\E(u,A)&\leq C\log |\log \ep|(N_\good^{\ep,\de_1}+C).\notag
\end{align}
Arguing as in the previous step, we deduce that $N_\good^{\ep,\de_1}\leq C$ for some constant $C>0$ independent of $\ep$. Inserting this into \eqref{bad1}, we find
$$
\sum_{i\in I_\bad^{\ep,\de_1}}|\ga_i^\ep|\leq C\frac{\log|\log\ep|}{|\log\ep|^{1-\alpha}}.
$$
Finally, by Lemma \ref{lem:controllength}, Condition \ref{nondegencond}, and definition of $I_\good^{\ep,\de_1}$, we have that, for any $i\in I_\good^{\ep,\de_1}$,
$$
\|\ga_i^\ep-\ga_0\|_*\leq C\de_1 \quad \mathrm{and}\quad \left||\ga_i^\ep|-|\ga_0| \right|\leq C\de_1.
$$
Letting $N_0\colonequals N_\good^{\ep,\de_1}$ and $\tga=\sum_{i\in I_\bad^{\ep,\de_1}}\ga_i^\ep$, the proof of Theorem \ref{teo:boundedvorticity} is finished.

\end{proof}
\section{The first critical field and proof of Theorem \ref{teo:Hc1}}\label{sechc1}
We begin this section by providing a lower bound for the free energy with an $O(1)$ error, when we know that the vorticity is essentially concentrated in finitely many lines around $\ga_0$.

\begin{proposition}\label{prop:lowerbound}
Let us assume that $\ga_0:[a,b]\to \RR^3$ is straight and that $\partial\Omega$ is negatively curved or flat on the endpoints of $\ga_0$. Then, under the assumptions of Theorem \ref{teo:boundedvorticity}, we have
$$
F_\ep(u,A)\geq \pi N_0|\ga_0||\log \ep|-C_0,
$$
where $C_0>0$ is a constant that does not depend on $\ep$.
\end{proposition}

\begin{proof}
Let $(\u, \A)$ (depending on $\ep$) be such that $GL_\ep(\u,\A)\leq h_\ex J_0^2$.
We let $U$ be the regular tubular neighborhood of thickness $c_0>0$ of $\ga_0$ and consider the map $\psi$ defined in the proof of Theorem \ref{teo:boundedvorticity}. Since $\ga_0$ is straight, the coarea formula yields
$$
F_\ep(u,A)\geq \int_a^b \left(\int_{\Sigma_z} e_\ep(u,A)dxdy\right)dz,
$$ 
where  $e_\ep$ is the energy density corresponding to $F_\ep$ and $\Sigma_z$ corresponds to the slice $\psi^{-1}(z)\cap U$. Let us observe that since we assume that $\partial\Omega$ is negatively curved on the endpoints of $\ga_0$,  the size of the slices is nondecreasing as one approaches  the endpoints. Without this assumption, the slices close to the endpoints of $\ga_0$ would shrink, which would not allow the following argument to work.

Let us now consider slices of the $1$-currents $\mu(u,A)$ and $\nu_\ep$ by the function $\zeta(x,z)=z$. The slice of $\mu(u,A)$ in $U$ by $\zeta^{-1}(z)$ is a $0$-current supported in $\zeta^{-1}(z)$, which we will denote by $\mu^z(u,A)$ (in the geometric measure theory literature, it is usually denoted by $\langle \mu(u,A),\zeta,z\rangle$). 
For almost every $z$, $\mu^z(u,A)$ is the two dimensional Lebesgue measure on the horizontal plane $\zeta^{-1}(z)$, weighted by vorticity of $(u,A)$ in the $x$,$y$ variables. The slice $\langle \nu_\ep,\zeta,z\rangle$ which we denote $\nu_\ep^z$ corresponds, still for almost every $z$, to a sum of multiples of Dirac masses at points on the plane $\zeta^{-1}(z)$.

%\ca{We remark that, for almost every $z$, $\mu^z(u,A)$ just is the vorticity of $(u,A)$ with respect to the first two space variables (with fixed $z$). %{\color{red} *** This probably is the easiest way to say that the slice of the vorticity is just the 2D vorticity for fixed $z$. Please let me know if you think it is clear***}
%Analogously, the slice of $\nu_\ep$ in $U$ by $\zeta^{-1}(z)$ is a $0$-current supported in $\zeta^{-1}(z)$, which we will denote by $\nu_\ep^z$. Let us observe that by construction of $\nu_\ep$, for almost every $z\in (a_\ep,b_\ep)$, its slice $\nu_\ep^z$ corresponds to a sum of multiples of Dirac masses.
%}

By taking $n$ large enough in Theorem \ref{theorem:epslevel} (and thus a large $q$ with $\alpha=1/2$), according to Lemma \ref{lemma:vorticityflat}, we have
$$
\|\mu(u,A) -\nu_\ep\|_{\F(\Omega_\ep)}\leq C|\log\ep|^{-4}.
$$
Then, letting $a_\ep=a+|\log\ep|^{-(q+1)}$ and $b_\ep=b-|\log\ep|^{-(q+1)}$, by \cite{Fed}*{4.3.1}, we deduce that
$$
\int_{a_\ep}^{b_\ep} \|\mu^z(u,A)-\nu_\ep^z\|_{F(\Sigma_z)}dz\leq C \|\mu(u,A)-\nu_\ep\|_{F(\Omega_\ep)}\leq C|\log\ep|^{-4},
$$
where $C>0$ is a constant that depends on $\ga_0$.

We let 
$$
Z_\good\colonequals \left\{z\in (a_\ep,b_\ep) \ | \ \|\mu^z(u,A)-\nu_\ep^z\|_{F(\Sigma_z)}\leq |\log\ep|^{-2}\right\}.
$$
Then, using the previous estimate, we deduce that
\begin{equation}\label{estZ}
|(a_\ep,b_\ep)\setminus Z_\good|\leq C|\log\ep|^{-2}.
\end{equation}
Let us now define 
$$
Z\colonequals \left\{ z\in Z_\good \ \vline  \ \int_{\Sigma_z}e_\ep (u,A)dxdy\geq \pi \left(N_0+\frac12\right)|\log \ep|\right\}.
$$
Let $z\in Z_\good\setminus Z$. Then, from the ``ball construction" with final radius $r=|\log\ep|^{-3}$, see \cite{SanSerBook}*{Chapter 4}, we get a collection $\mathcal B= \{B_i(a_i,r_i) \}_i$ of $k$ balls such that $r=\sum_i r_i$,
$$
\int_{\Sigma_z}e_\ep(u,A)dxdy\geq \pi D \left(|\log \ep|+\log r/D-C\right),
$$
where $D=\sum_i|d_i|$ is bounded independently of $\ep$ (since $N_0$ is bounded independently of $\ep$) and, for a.e. $z$,
$$
\left\|\mu^z(u,A)-2\pi\sum_{i=1}^k d_i \delta_{a_i}\right\|_{\F(\Sigma_z)}\leq C|\log\ep|^{-2}.
$$
We use the decomposition of $\nu_\ep$ defined in the proof of Theorem \ref{teo:boundedvorticity}, in order to write
$$
\nu_\ep^z=\nu_0^z+\nu_1^z,
$$
where $\nu_0^z$ (resp. $\nu_1^z$) corresponds to the slice of the ``good'' (resp. ``bad'') curves by $\zeta$. We recall that $\tga$ in the statement of this theorem corresponds to the sum in the sense of currents of the previously mentioned ``bad'' curves. We then have, for a.e. $z$,
$$
\nu_0^z=2\pi\sum_{i} d_i^0\delta_{b_i}\quad \mbox{and}\quad \nu_1^z=2\pi\sum_{i}d_i^1\delta_{c_i},
$$
where $\sum_{i}d_i^0=N_0$ and all the points $b_i$ are located at distance less or equal than a negative power of $|\log\ep|$ to the origin. We denote by $\mathbb L$ the minimal connection joining the collection of points $c_i$ (with their associated degree $d_i^1$) between each other, allowing connections to the boundary of $\Sigma_z$. Observe that, for a.e. $z$,
$$
\|\nu_1^z\|_{\F(\Sigma_z)}\leq |\mathbb L|.
$$
Moreover, choosing $\alpha=\frac14$ in Theorem \ref{teo:boundedvorticity}, we have
$$
|\mathbb L|\leq \sum_{i\in I_\bad} |\ga_i|=\left|\tga\right|\leq C\frac{\log|\log\ep|}{|\log\ep|^\frac34},
$$
which directly follows from the definition of $\mathbb L$. We are connecting the traces of the ``bad'' curves contained in the slice (either between them or to the boundary), and, therefore, the length of the minimal connection must be less or equal than the total length of the ``bad'' curves.
We thus deduce that, for a.e. $z$,
$$
\left\| 2\pi \sum_{i=1}^k d_i\delta_{a_i}-\nu_0^z \right\|_{\F(\Sigma_z)}\leq  C\frac{\log|\log\ep|}{|\log\ep|^\frac34}.
$$
Since $\sum_i |d_i|\leq N_0$ thanks to the upper bound for $z\in Z_\good\setminus Z$, and $\sum_{i} d_i^0=N_0$, we deduce that $d_i\geq 0$ and $\sum_i d_i=N_0$. Moreover, all the points $a_i$ with $d_i\neq 0$ are located at distance less or equal than a negative power of $|\log\ep|$ to the origin (since the points $b_i$ satisfy this).

We now let the balls grow again, following \cite{SanSerBook}*{Chapter 4}, choosing $c_0/100$ as final radius. This yields a new collection of balls, such that
$$
\int_{\Sigma_z}e(u,A)dxdy\geq \pi D \left(|\log \ep|-C\right),
$$
where the total degree $D$ remains the same as the before, that is, $D=N_0$, and $C$ is a constant that does not depend on $z$.

Finally, we have that
\begin{align*}
F(u,A)&\geq \int_a^b \left(\int_{\Sigma_z} e(u,A)dxdy\right)dz\\
&\geq \int_{a_\ep}^{b_\ep} \left(\int_{\Sigma_z} e(u,A)dxdy\right)dz\\
&\geq \int_{Z_\good\setminus Z}\left(\int_{\Sigma_z} e(u,A)dxdy\right)dz+ \int_{ Z}\left(\int_{\Sigma_z} e(u,A)dxdy\right)dz\\
&\geq \pi |Z_\good \setminus Z| N_0 \left(|\log \ep|-C\right) +\pi |Z| \left(N_0 +\frac12\right)|\log\ep|\\
&\geq \pi |Z_\good|N_0 \left(|\log \ep|-C\right).
\end{align*}
Using \eqref{estZ}, we deduce that 
$$|\ga_0|-|Z_\good|\leq |a-a_\ep|+|b-b_\ep|-|(a_\ep,b_\ep)\setminus Z_\good|=o(|\log\ep|),$$
and therefore
$$
F_\ep(u,A)\geq \pi |\ga_0|N_0\left(|\log \ep|-C\right),
$$
and the proposition is thus proved.
\end{proof}

We now provide a proof for Theorem \ref{teo:Hc1}. 
\begin{proof}[Proof of Theorem \ref{teo:Hc1}]
First we note that since Condition \ref{nondegencond} is assumed to hold, Theorem \ref{teo:boundedvorticity} applies. Let $(\u,\A)$ minimize $GL_\ep$.

\medskip\noindent
{\bf Step 1: Proving that $N_0=0$.} Let us assume towards a contradiction that $N_0>0$.  

Proposition \ref{prop:energysplitting} yields
\begin{equation*}%\label{LB}
GL_\ep(\u,\A)\geq h_{\ex}^2 J_0+F_\ep(u,A)-h_{\ex}\int_\Omega \mu(u,A)\wedge B_0+o(\ep^\frac12),
\end{equation*}
where $(\u,\A)=(e^{ih_\ex\phi_0}u,h_\ex A_0+A)$.
Using that \eqref{majoF} holds, we may then apply Proposition \ref{prop:lowerbound} and the vorticity estimate in Theorem \ref{teo:boundedvorticity} to obtain
$$
F_\ep(u,A)-h_{\ex}\int_\Omega \mu(u,A) \wedge B_0 \geq  \pi N_0|\ga_0||\log\ep| -2\pi h_{\ex}\int_\Omega \(\sum_{i=1}^{N_0}\ga_i+\tga\) \wedge B_0-2C_0,
$$
where $C_0$ is the constant that appears in the statement of Proposition \ref{prop:lowerbound}. Observe that 
$$
h_\ex\int_\Omega \(\sum_{i=1}^{N_0}\ga_i+\tga\)\wedge B_0= h_\ex\sum_{i=1}^{N_0} \pr{B_0,\ga_i}+ h_\ex \pr{B_0,\tga}.
$$ 
By choosing $\alpha=\frac14$ in Theorem \ref{teo:boundedvorticity}, we have 
$$
|\tga|\leq C\frac{\log|\log\ep|}{|\log\ep|^\frac34}.
$$
Therefore, by the isoperimetric inequality (see \cite{Rom2}*{Remark 4.1}), we deduce that
$$
\left|\pr{B_0,\tga}\right|\leq C\left|\tga\right|^2\leq C \frac{(\log|\log \ep|)^2}{|\log\ep|^\frac32}.
$$
Thus,
$$
h_\ex\int_\Omega \(\sum_{i=1}^{N_0}\ga_i+\tga\)\wedge B_0= h_\ex\sum_{i=1}^{N_0} \pr{B_0,\ga_i}+o(1).
$$
Moreover, for $i=1,\dots,N_0$, we have
$$
\|\ga_i-\ga_0\|_*\leq C\left(\frac{\log\lep}{\lep^\frac14} \right)^\frac1N\leq \de,
$$
where the last inequality holds for every sufficiently small $\ep>0$. Here $\de$ is the constant that appears in \eqref{hypde}. Therefore
$$
\pr{B_0,\ga_0}\geq \pr{B_0,\ga_i}\quad \mathrm{for} \ i=1,\dots,N_0.
$$ 
Thus 
\begin{multline*}%\label{nu}
h_\ex\int_\Omega \(\sum_{i=1}^{N_0}\ga_i+\tga\)\wedge B_0= h_\ex\sum_{i=1}^{N_0} \pr{B_0,\ga_i}+o(1)\\
\leq  h_\ex\sum_{i=1}^{N_0} \langle B_0,\ga_0\rangle+o(1)= h_\ex N_0\langle B_0,\ga_0\rangle+o(1).
\end{multline*}
Hence
$$
F_\ep(u,A)-h_{\ex}\int_\Omega \mu(u,A) \wedge B_0 \geq \pi N_0|\ga_0| \left(|\log\ep| - 2\R_0 h_\ex \right)-2C_0.
$$
Writing $h_{\ex}=H_{c_1}^0-K_0$ with $H_{c_1}^0=\dfrac1{2\R_0}|\log \ep|$, we get
$$
GL_\ep(\u,\A)\geq h_{\ex}^2J_0+2\pi N_0 |\ga_0|\R_0 K_0-2C_0.
$$
Combining with \eqref{boundJ0}, we deduce that
$$
0\geq 2\pi N_0 |\ga_0|\R_0 K_0-2C_0.
$$
We therefore reach a contradiction provided $K_0$ is large enough. Hence, $N_0=0$.

\medskip\noindent
{\bf Step 2. Applying a clearing out result.}
It suffices to reproduce Step 2 of the proof of Theorem 2 in \cite{Rom2} to conclude that $\|1-|\u|\|_{L^\infty} =o(1)$ for $(\u, \A)$ a minimizer.
\end{proof}
As a direct consequence of the previous result, we obtain an improved lower bound for $H_{c_1}$.

\begin{corollary}
There exist constants $\ep_0,K_0>0$ such that for any $\ep<\ep_0$ we have 
$$
H_{c_1}^0-K_0\leq H_{c_1}.
$$
\end{corollary}
This combined with \cite{Rom2}*{Theorem 1.3} yields that
$$
H_{c_1}=H_{c_1}^0+O(1).
$$
\appendix
\section{Proof of the vorticity estimate for the flat norm}\label{appendix}
In this section we provide a proof of Lemma \ref{lemma:vorticityflat}, which consists of a modification of the proof of \eqref{EstimateJ0} for $\gamma=1$ in \cite{Rom}*{Section 8}. For the reader convenience, we will start by recalling some of the key elements of the construction of the vorticity approximation $\nu_\ep$.

\subsection{Choice of grid}
Let us fix an orthonormal basis $(e_1,e_2,e_3)$ of $\RR^3$ and consider a grid $\GG=\GG(a,\de)$ given by the collection of closed cubes $\CC_i\subset \RR^3$ of side-length $\de=\de(\ve)$ (conditions on this parameter are given in the lemma below).
In the grid we use a system of coordinates with origin in $a \in \Omega$ and orthonormal directions $(e_1,e_2,e_3)$. 
From now on we denote by $\RRR_1$ (respectively $\RRR_2$) the union of all edges (respectively faces) of the cubes of the grid. We have the following lemma, taken from \cite{Rom}*{Section 2}.

\begin{lemma}[Choice of grid]\label{Lemma:Grid} 
For any $\gamma\in(-1,1)$ there exist constants $c_0(\gamma),c_1(\gamma)>0$, $\de_0(\Omega)\in(0,1)$ such that, for any $\ve,\de>0$ satisfying 
$$
\ve^{\frac{1-\gamma}2}\leq c_0\quad \mathrm{and}\quad c_1\ve^{\frac{1-\gamma}4}\leq \de \leq \de_0,
$$
if $(u,A)\in H^1(\Omega,\C)\times H^1(\Omega,\RR^3)$ is a configuration such that $F_\ve(u,A)\leq \ve^{-\gamma}$ then there exists $b_\ve\in \Omega$ such that the grid $\GG(b_\ve,\delta)$ satisfies
\begin{subequations}\label{propGrid}
\begin{equation}\label{prop1Grid}
|u_\ve|>5/8\quad \mathrm{on}\ \RRR_1(\GG(b_\ve,\de))\cap \Omega,
\end{equation}
\begin{equation*}\label{prop2Grid}
\int\limits_{\RRR_1(\GG(b_\ve,\delta))\cap \Omega}e_\ve(u,A)d\H^1\leq  C \de^{-2}F_\ve(u,A),
\end{equation*}
\begin{equation*}\label{prop3Grid}
\int\limits_{\RRR_2(\GG(b_\ve,\delta))\cap \Omega}e_\ve(u,A)d\H^2\leq  C \de^{-1}F_\ve(u,A),
\end{equation*}
\end{subequations}
where $C$ is a universal constant, and where hereafter $\H^d$ denotes the $d$-dimensional Hausdorff measure for $d\in \N$ and $e_\ep$ denotes the energy density corresponding to $F_\ep$. 
\end{lemma}

From now on we drop the cubes of the grid $\GG(b_\ve,\de)$ given by the previous lemma, whose intersection with $\RR^3\setminus \Omega$ is non-empty. We also define 
\begin{equation}
\Theta\colonequals \Omega \setminus \cup_{\CC_l\in \GG} \CC_l\quad\mathrm{and}\quad \partial \GG \colonequals \partial \left(\cup_{\CC_l\in \GG} \CC_l \right)\label{unioncubes}.
\end{equation}
Observe that, in particular, $\partial \Theta=\partial \GG \cup \partial \Omega$.

We remark that $\GG(b_\ve,\de)$ carries a natural orientation. The boundary of every cube of the grid will be oriented accordingly to this orientation. Each time we refer to a face $\omega$ of a cube $\CC$, it will be considered to be oriented with the same orientation of $\partial\CC$. If we refer to a face $\omega\subset \partial\GG$, then the orientation used is the same of $\partial\GG$.

\subsection{2D vorticity estimate}
Given a two dimensional Lipschitz domain $\omega \subset \Omega$, we let $(s,t,0)$ denote coordinates in $\RR^3$ such that $\omega\subset \{(s,t,0) \in \Omega\}$. We define $\mu_\ve\colonequals \mu_\ve(u,A)[\partial_s,\partial_t]$, and write $\mu_{\ve,\omega}$ its slice (in the same sense as in the proof of Proposition \ref{prop:lowerbound}) supported in $\{(s,t,0) \in \Omega\}$.
We have the following 2D vorticity estimate (see \cite{Rom}*{Corollary 4.1}).
\begin{lemma}\label{cor:2dVortEstimate}
Let $\gamma\in(0,1)$ and assume that $(u,A)\in H^1(\Omega,\C)\times H^1(\Omega,\RR^3)$ is a configuration such that $F_\ve(u,A)\leq \ve^{-\gamma}$, so that by Lemma \ref{Lemma:Grid} there exists a grid $\GG(b_\ve,\de)$ satisfying \eqref{prop1Grid}. Then there exists $\ve_0(\gamma)$ such that, for any $\ve<\ve_0$ and for any face $\omega\subset \RRR_2(\GG(b_\ve,\de))$ of a cube of the grid $\GG(b_\ve,\de)$, letting $\{U_{j,\omega}\}_{j\in J_\omega}$ be the collection of connected components of $\{x\in\omega \ | \ |1-|u(x)||\geq 1/2\}$ and
$\{S_{i,\omega}\}_{i\in I_\omega}$ denote the collection of connected components of $\{x\in\omega \ | \ |u_\ve(x)|\leq 1/2\}$ whose degree $d_{i,\omega}\colonequals \mathrm{deg}(u/|u|,\partial S_{i,\omega})\neq 0$, we have
\begin{multline*}
\left\| \mu_{\ve,\omega} -2\pi \sum_{i\in I_\omega}d_{i,\omega}\delta_{a_{i,\omega}} \right\|_{C^{0,1}(\omega)^*}\leq \\ C\max(r_\omega,\ve)\left(1+\int_\omega e_\ve(u,A)d\H^2+\int_{\partial \omega} e_\ve(u,A)d\H^1\right),
\end{multline*}
where $a_{i,\omega}$ is the centroid of $S_{i,\omega}$, $r_\omega\colonequals \sum_{j\in J_\omega}\mathrm{diam}(U_{j,\omega})$, and $C$ is a universal constant.
\end{lemma}

In view of the previous corollary, it is important to bound from above $r_\omega$. For that, we borrow from \cite{Rom}*{Remark 4.1}, the estimate
\begin{equation}
\label{radiusGrid}
\sum_{\omega \subset \RRR_2(\GG(b_\ve,\de))} r_\omega\leq C\ve \int\limits_{\RRR_2(\GG(b_\ve,\de))} e_\ve(u,A)d\H^2 \leq C \ve \delta^{-1}F_\ve(u,A),
\end{equation}
where $\sum_{\omega\subset \RRR_2(\GG(b_\ve,\de))}$ denotes the sum over all the faces $\omega$ of cubes of the grid $\GG(b_\ve,\de)$.

\subsection{Construction of the vorticity approximation\label{subsec:vorticity}}
For the purpose of this appendix, it is not necessary to remind the full construction of the vorticity approximation $\nu_\ep$ that appears in the statement of Theorem \ref{theorem:epslevel}. Nevertheless, let us next mention a few  of the main ingredients. 

Given $\gamma\in (0,1)$ and a configuration $(u,A)\in H^1(\Omega,\C)\times H^1(\Omega,\RR^3)$ such that $F_\ve(u,A)\leq \ve^{-\gamma}$, Lemma \ref{Lemma:Grid} provides a grid $\GG(b_\ve,\de)$ satisfying \eqref{propGrid}. On the face $\omega$ of each cube $\CC_l$ of the grid, Corollary \ref{cor:2dVortEstimate} gives the existence of points $a_{i,\omega}$ and integers $d_{i,\omega}\neq 0$ such that
\begin{equation}\label{apmu}
\mu_{\ve,\omega}\approx 2\pi \sum_{i\in I_\omega} d_{i,\omega}\delta_{a_{i,\omega}}.
\end{equation}
The vorticity approximation $\nu_\ep$ is then constructed in such a way that its restriction to the face $\omega$ of any cube of the grid coincides with the right-hand side of \eqref{apmu}. The key for this is that, since $\partial\mu(u,A)=0$ relative to any cube of the grid $\CC_l$, we have
$$
\sum_{\omega\subset \partial \CC_l} \sum_{i\in I_\omega} d_{i,\omega}=0,
$$
which allows to define the vorticity approximation in any cube of the grid as the minimal connection associated to the collection of points where the vortices are located on the boundary of the cube. Analogously, since $\partial \mu (u,A)=0$ relative to $\partial \GG$, we have 
$$
\sum_{\omega\subset \partial \GG} \sum_{i\in I_\omega} d_{i,\omega}=0,
$$
which allows to define the vorticity approximation in $\Theta$ as the minimal connection through the boundary $\partial\Omega$ associated to the collection of points where the vortices are located on $\partial\GG$. The detailed construction can be found in \cite{Rom}*{Section 5}.

\subsection{Proof of the vorticity estimate for the flat norm}
We now have all the ingredients to provide a proof for the vorticity estimate for the flat norm.

%%%%%%%%%%%%%%%%%%%%%%%%%%%%%%%%%%%%

\begin{proof}
As we previously mentioned, the proof consists of a modification of the proof of \eqref{EstimateJ0} for $\gamma=1$ in \cite{Rom}*{Section 8}. We will present it in the language of vector calculus, so we will work with vector fields and functions (that in very few places are identified with differential forms). 

\medskip\noindent
{\bf Step 1:  Hodge decomposition and elliptic regularity estimates. } 
Let $\GG(b_\ep,\de)$ be the grid of cubes given by Lemma \ref{Lemma:Grid} with $\de=|\log\ep|^{-q}$, which implies that $\Omega_\ep\subset (\Omega\setminus \Theta)$.

We consider a smooth vector field $X\in C_0^\infty(\Omega_\ep,\RR^3)$ with $\max\{\|X\|_\infty,\|\curl X\|_\infty\}\leq 1$. We extend $X$ by $0$ outside $\Omega_\ep$. By \cite {AlaBroMon}*{Lemma 7.1} we can decompose it as 
\begin{equation}\label{hodge}
X=\curl B_X+ \nabla \phi_X\quad \mbox{in }B(0,R),
\end{equation}
where $B(0,R)$ is a ball such that $\Omega$ is compactly contained in it, $B_X$ is a divergence free vector field such that $B_X\times \vec \nu =0$ on $\partial B(0,R)$ and $\phi_X$ is a function defined up to a constant, so that we may assume that $\phi_X$ has zero average over $\Omega$, that is $\int_{\Omega} \phi_X=0$. Moreover,
\begin{equation}\label{estBA}
\|B_X\|_{W^{1,2}(B(0,R),\RR^3)}\leq C\|X\|_{L^2(B(0,R),\RR^3)}.
\end{equation}

On the other hand, we have 
$$
-\Delta B_X=\curlcurl B_X=\curl X\quad \mbox{in } B(0,R),
$$
and since $\Omega$ is compactly contained in $B(0,R)$, interior elliptic regularity yields that $B_X\in C^{1,\alpha}(\Omega,\RR^3)$ and 
$$
\|B_X\|_{C^{1,\alpha}(\Omega,\RR^3)}\leq C\( \|\curl X\|_{C^0(B(0,R),\RR^3)} +\|B_X\|_{L^2(B(0,R),\RR^3)}\)
$$
for any $\alpha\in (0,1)$. By combining with \eqref{estBA} and using that $X$ vanishes outside $\Omega_\ep$, we find
\begin{equation}\label{C1}
\|B_X\|_{C^{1,\alpha}(\Omega,\RR^3)}\leq C\( \|X\|_{C^0(\Omega_\ep,\RR^3)} +\|\curl X\|_{C^0(\Omega_\ep,\RR^3)}\).
\end{equation}

\medskip
On the other hand, since 
$$
\Delta \phi_X=\diver X\quad \mbox{in }B(0,R),
$$
by elliptic regularity we deduce that $\phi_X$ is smooth, since $X$ is smooth. Moreover, by combining \eqref{hodge} with \eqref{C1}, we have
$$
|\phi_X|_{\mathrm{Lip}(\Omega)}
=\|\nabla \phi_X\|_{C^0(\Omega,\RR^3)}\leq C\( \|X\|_{C^0(\Omega_\ep,\RR^3)} +\|\curl X\|_{C^0(\Omega_\ep,\RR^3)}\),
$$
where $|\phi_X|_{\mathrm{Lip}(\Omega)}$ denotes the Lipschitz seminorm of $\phi_X$ in $\Omega$. Besides, since $\int_{\Omega}\phi_X=0$, we have that
\begin{equation}
\label{C3}
\|\phi_X\|_{C^{0,1}(\Omega)}\leq C\|\phi_X\|_{\mathrm{Lip}(\Omega)}
\leq C\( \|X\|_{C^0(\Omega_\ep,\RR^3)} +\|\curl X\|_{C^0(\Omega_\ep,\RR^3)}\).
\end{equation}

\medskip\noindent
{\bf Step 2:  Vorticity estimate of each term of the Hodge decomposition. } 
Let us now write
\begin{multline}\label{mu1}
\int_{\Omega\setminus \Theta} (\mu(u,A)-\nu_\ep)\wedge X\\
=\int_{\Omega\setminus \Theta} (\mu(u,A)-\nu_\ep)\wedge \curl B_X+\int_{\Omega\setminus \Theta} (\mu(u,A)-\nu_\ep)\wedge \nabla \phi_X.
\end{multline}
First, by integrating by parts, we have 
$$
\int_{\Omega\setminus \Theta} \mu(u,A)\wedge \nabla \phi_X=\sum_{\omega\subset \partial \GG}\int_{\omega}\mu_{\ve,\omega}\phi_X.
$$
%Let us remark that in the term on the right-hand side $\mu(u,A)$ must be understood as a $0$-current supported on $\partial\GG$, which one obtains by slicing the 3D vorticity. The same applies in the next equality, which involves slices of $\nu_\ep$. The details can be found in \cite{Rom}*{Section 8}. %{\color{red} Is this comment clear enoguh?}

On the other hand, by construction of $\nu_\ep$ (recall in particular that $\nu_\ep$ is a sum in the sense of currents of simple Lipschitz curves that do not have endpoints in $\Omega\setminus \Theta$), we have
$$
\int_{\Omega\setminus \Theta} \nu_\ep\wedge \nabla \phi_X=\sum_{\omega\subset \partial \GG}\int_{\omega}2\pi\sum_{i\in I_\omega}d_{i,\omega}\delta_{a_{i,\omega}}\phi_X.
$$
%Now, since $\partial \GG$ is a surface without boundary, a 2D vorticity estimate gives
%$$
%\left| \int_{\partial \GG}(\mu(u,A)-\nu_\ep)\ \phi_X\right|\leq C |\phi_X|_{\mathrm{Lip}(\partial\GG)} r_\GG \(F_\ep(u,A,\partial \GG)+1\),
%$$
%where $r_\GG$ is defined in \cite{Rom}*{Remark 4.1} and where hereafter $F_\ep(u,A,S)$ denotes the integral over the surface $S$ of the energy density corresponding to $F_\ep$ with respect to the 2-dimensional Hausdorff measure. To prove this one can essentially apply the same argument as in \cite{Rom}*{Theorem 4.1} on each face of a cube contained in $\partial \GG$ and then sum over all faces, which makes all the boundary terms (i.e. over the edges of each face) go away (recall that $\partial \GG$ does not have a boundary). Thanks to this, only the Lipschitz seminorm of $\phi_X$ in $\partial\GG$ remains in the right-hand side of the vorticity estimate.
Combining the previous two estimates, we have
$$
\int_{\Omega\setminus \Theta} (\mu(u,A)-\nu_\ep)\wedge \nabla \phi_X=\sum_{\omega\subset \partial \GG}\int_{\omega}\left(\mu_{\ve,\omega}-2\pi\sum_{i\in I_\omega}d_{i,\omega}\delta_{a_{i,\omega}}\right)\phi_X,
$$
and therefore, by combining with Lemma \ref{cor:2dVortEstimate}, we find
\begin{multline*}
\left|\int_{\Omega\setminus \Theta} (\mu(u,A)-\nu_\ep)\wedge \nabla \phi_X \right|\leq \\
C\sum_{\omega\subset \partial\GG}\max(r_\omega,\ve)\left(1+\int_\omega e_\ve(u_\ve,A_\ve)d\H^2+\int_{\partial \omega} e_\ve(u_\ve,A_\ve)d\H^1\right)\|\phi_X\|_{C^{0,1}(\partial\GG)}.
\end{multline*}

Therefore, by the choice of grid (see Lemma \ref{Lemma:Grid}) and \eqref{radiusGrid}, we have
\begin{equation}
\label{mu2}
\left|\int_{\Omega\setminus \Theta} (\mu(u,A)-\nu_\ep)\wedge \nabla \phi_X \right|\leq C \|\phi_X\|_{C^{0,1}(\partial\GG)} \ep \de^{-3}\(F_\ep(u,A)^2+1\).
\end{equation}

We now consider a cube $\CC_l\in \GG(b_\ep,\delta)$ and define $(\curl B_X)_l \colonequals \int_{\CC_l}\curl B_X$. Since $\curl B_X\in C^{0,\alpha}(\Omega)$ for any $\alpha\in (0,1)$, we have that
\begin{equation}\label{est1b}
\|\curl B_X -(\curl B_X)_l\|_{C^0(\CC_l,\RR^3)}\leq \de^\alpha \|\curl B_X\|_{C^{0,\alpha}(\CC_l,\RR^3)}.
\end{equation}
Moreover, 
\begin{multline}\label{b1}
\left| \int_{\CC_l} (\mu(u,A)-\nu_\ep)\wedge \curl B_X\right|\leq
\left|\int_{\CC_l} (\mu(u,A)-\nu_\ep)\wedge \(\curl B_X-(\curl B_X)_l\)\right|\\+\left|\int_{\CC_l} (\mu(u,A)-\nu_\ep)\wedge (\curl B_X)_l\right|.
\end{multline}
Using \eqref{est1b}, we deduce that
\begin{multline}
\label{b2}
\left|\int_{\CC_l} (\mu(u,A)-\nu_\ep)\wedge \(\curl B_X-(\curl B_X)_l\)\right|\\ \leq \de^\alpha\|\mu(u,A)-\nu_\ep\|_{C^0(\CC_l,\RR^3)^*}\|\curl B_X\|_{C^{0,\alpha}(\CC_l,\RR^3)}.
\end{multline}
On the other hand, since $(\curl B_X)_l$ is a constant, there exists a function $f_l$ such that 
$$
(\curl B_X)_l=\nabla f_l,\quad \int_{\CC_l}f_l=0.
$$ 
In particular
\begin{equation}\label{estcurlBl}
\|f_l\|_{C^{0,1}(\CC_l)}\leq |(\curl B_X)_l|.
\end{equation}
Arguing as before, by an integration by parts and construction of $\nu_\ep$, we have
\begin{multline*}
\int_{\CC_l} (\mu(u,A)-\nu_\ep)\wedge (\curl B_X)_l=\int_{\CC_l} (\mu(u,A)-\nu_\ep)\wedge \nabla f_l\\
=\sum_{\omega\subset \partial \CC_l}\int_{\omega}\left(\mu_{\ve,\omega}-2\pi\sum_{i\in I_\omega}d_{i,\omega}\delta_{a_{i,\omega}}\right)f_l,
\end{multline*}
and therefore, by combining with Lemma \ref{cor:2dVortEstimate}, we find
\begin{multline*}
\left|\int_{\CC_l} (\mu(u,A)-\nu_\ep)\wedge \nabla f_l  \right|\leq \\
C\sum_{\omega\subset \partial\CC_l}\max(r_\omega,\ve)\left(1+\int_\omega e_\ve(u_\ve,A_\ve)d\H^2+\int_{\partial \omega} e_\ve(u_\ve,A_\ve)d\H^1\right)\|f_l\|_{C^{0,1}(\partial\CC_l)},
\end{multline*}
which combined with \eqref{estcurlBl}, yields
\begin{multline*}
\left|\int_{\CC_l} (\mu(u,A)-\nu_\ep)\wedge (\curl B_X)_l  \right|\leq \\
C\sum_{\omega\subset \partial\CC_l}\max(r_\omega,\ve)\left(1+\int_\omega e_\ve(u_\ve,A_\ve)d\H^2+\int_{\partial \omega} e_\ve(u_\ve,A_\ve)d\H^1\right)|(\curl B_X)_l|.
\end{multline*}
Plugging in this and \eqref{b2} into \eqref{b1}, yields 
\begin{multline*}
\left|\int_{\CC_l} (\mu(u,A)-\nu_\ep)\wedge \curl B_X \right|\leq \\
C\|\curl B_X\|_{C^{0,\alpha}(\CC_l,\RR^3)}\Bigg(\de^\alpha\|\mu(u,A)-\nu_\ep\|_{C^0(\CC_l,\RR^3)^*}\\
+\sum_{\omega\subset \partial\CC_l}\max(r_\omega,\ve)\left(1+\int_\omega e_\ve(u_\ve,A_\ve)d\H^2+\int_{\partial \omega} e_\ve(u_\ve,A_\ve)d\H^1\right)\Bigg) .
\end{multline*}
Thus, by summing over cubes and using Lemma \ref{Lemma:Grid} and \eqref{radiusGrid}, we obtain 
\begin{multline}
\label{my}
\left|\int_{\Omega\setminus\Theta} (\mu(u,A)-\nu_\ep)\wedge \curl B_X \right|\\
\leq  C\|\curl B_X\|_{C^{0,\alpha}(\Omega\setminus\Theta,\RR^3)} \left(
\de^\alpha\|\mu(u,A)-\nu_\ep\|_{C^0(\Omega\setminus\Theta,\RR^3)^*}+ \ep \de^{-3}\(F_\ep(u,A)^2+1\)\right).
\end{multline}

On the other hand, by \cite{Rom}*{Lemma 8.1}, we have
$$
\|\mu(u,A)\|_{C^0(\Omega,\RR^3)^*}\leq CF_\ep(u,A).
$$
Moreover, from the lower bound in Theorem \ref{theorem:epslevel}, we have
$$
\|\nu_\ep\|_{C^0(\Omega,\RR^3)^*}\leq CF_\ep(u,A)\lep^{-1}.
$$
Hence
$$
\|\mu(u,A)-\nu_\ep\|_{C^0(\Omega,\RR^3)^*}\leq CF_\ep(u,A).
$$
By inserting in \eqref{my}, we find
\begin{multline}
\label{mu3}
\left|\int_{\Omega\setminus\Theta} (\mu(u,A)-\nu_\ep)\wedge \curl B_X \right|\\
\leq  C\|\curl B_X\|_{C^{0,\alpha}(\Omega\setminus\Theta,\RR^3)} \left(
\de^\alpha F_\ep(u,A)+ \ep \de^{-3}\(F_\ep(u,A)^2+1\)\right).
\end{multline}

Hence, by combining \eqref{mu1},\eqref{mu2}, and \eqref{mu3}, we are led to
\begin{multline*}
\left|\int_{\Omega\setminus\Theta} (\mu(u,A)-\nu_\ep)\wedge X \right|\leq \\
C\max\left\{\ep\de^{-3}(F_\ep(u,A)^2+1),\de^\alpha F_\ep(u,A)\right\}\(\|\curl B_X\|_{C^{0,\alpha}(\Omega\setminus \Theta,\RR^3)}+\|\phi_X\|_{C^{0,1}(\partial\GG)} \).
\end{multline*}
Finally, by inserting \eqref{C1} and \eqref{C3} in this inequality, we obtain 
\begin{multline*}
\left|\int_{\Omega\setminus\Theta} (\mu(u,A)-\nu_\ep)\wedge X \right|\leq \\
C\max\left\{\ep\de^{-3}(F_\ep(u,A)^2+1),\de^\alpha F_\ep(u,A)\right\}\( \|X\|_{C^0(\Omega_\ep,\RR^3)} +\|\curl X\|_{C^0(\Omega_\ep,\RR^3)}\),
\end{multline*}
and since 
$$
\max\{\|X\|_\infty,\|\curl X\|_\infty\}\leq 1,
$$
we get
$$
\left|\int_{\Omega\setminus\Theta} (\mu(u,A)-\nu_\ep)\wedge X \right|\leq \\
C\max\left\{\ep\de^{-3}(F_\ep(u,A)^2+1),\de^\alpha F_\ep(u,A)\right\}.
$$
Recalling that $\de=|\log\ep|^{-q}$, for any $\ep$ sufficiently small, we have that 
$$
\max\left\{\ep\de^{-3}(F_\ep(u,A)^2+1),\de^\alpha F_\ep(u,A)\right\}\leq \de^\alpha \(F_\ep(u,A)+1\),
$$ 
and therefore, recalling that $\Omega_\ep\subset (\Omega\setminus\Theta)$ and $X$ is supported in $\Omega_\ep$, we conclude that
$$
\left|\int_{\Omega_\ep} (\mu(u,A)-\nu_\ep)\wedge X \right|\leq \\
C\frac{(F_\ep(u,A)+1)}{|\log\ep|^{\alpha q}} 
$$
for any $\alpha \in (0,1)$. This concludes the proof. 
\end{proof}

\bibliography{referencesRSS}
\end{document}